\documentclass[a4paper,reqno,oneside]{amsart}

\usepackage[
style=numeric,
sorting=nty,
maxnames=99,
maxalphanames=5,
sortcites]{biblatex}

\usepackage[british]{babel}
\usepackage{csquotes}  
\numberwithin{equation}{section}
\usepackage{amssymb}
\usepackage[protrusion=true,expansion=true]{microtype}	
\usepackage{amsmath,amsfonts,amsthm} 
\usepackage[pdftex]{graphicx}	
\usepackage{url}
\usepackage{svg}
\usepackage[title,titletoc,page]{appendix} 
\usepackage{listings} 
\usepackage{color}
\usepackage{caption}
\usepackage{subcaption}
\usepackage{algorithm}
\usepackage{dsfont} 
\usepackage{bm}
\usepackage{upgreek}
\usepackage[normalem]{ulem}
\usepackage{enumerate}
\usepackage{mathtools}
\usepackage{lipsum}
\usepackage{adjustbox}
\usepackage{tikz}
\usepackage{circuitikz}
\usepackage{tcolorbox}
\usetikzlibrary{tikzmark,calc}
\usepackage{xargs}      
\usepackage[textsize=tiny]{todonotes}

\usepackage{svg}
\usepackage{nicematrix}

\usepackage[T1]{fontenc} 
\usepackage{lmodern} 
\rmfamily 
\DeclareFontShape{T1}{lmr}{b}{sc}{<->ssub*cmr/bx/sc}{}
\DeclareFontShape{T1}{lmr}{bx}{sc}{<->ssub*cmr/bx/sc}{}

\usepackage[left=2.5cm,right=2.5cm,top=3.5cm,bottom=3.5cm,footskip=1.5cm,headsep=1cm]{geometry}
\usepackage[%
pdfauthor={Davies, De Bruijn, Dupuy, Hiltunen},%
pdftitle={Toeplitz Eigenvector Decay},
pdfsubject={Toeplitz Eigenvector Decay},
pdfkeywords={Subwavelength resonances, evanescent modes, band gap, interface eigenmodes, Bloch theory, complex Brillouin zone, layer potentials}]{hyperref}

\usepackage{fancyhdr}
\fancypagestyle{plain}{%
\fancyhead[C]{}
\fancyfoot[C]{}

}

\fancypagestyle{nosection}{%
\fancyhead[CE]{}
\fancyhead[CO]{}
\fancyfoot[CE,CO]{\thepage}
}
\numberwithin{equation}{section}		
\numberwithin{figure}{section}			
\numberwithin{table}{section}			

\newcommand{\vect}[1]{\boldsymbol{\mathbf{#1}}}

\DeclareMathOperator{\C}{\mathbb{C}}

\newcommand{\N}{\mathbb{N}}

\newcommand{\R}{\mathbb{R}}

\newcommand{\p}{\partial}
\renewcommand{\epsilon}{\varepsilon}

\newcommand{\Bf}{\mathfrak{B}}

\renewcommand{\Re}{\mathrm{Re}}

\newcommand{\ds}{\displaystyle}

\renewcommand{\tilde}{\widetilde}
\renewcommand{\hat}{\widehat}
\renewcommand{\i}{\mathrm{i}}
\renewcommand{\d}{\,\mathrm{d}}
\renewcommand{\Re}{\mathfrak{Re}}
\renewcommand{\Im}{\mathfrak{Im}}

\newcommand{\Id}{\mathrm{Id}}

\usepackage[noabbrev,capitalize]{cleveref}
\crefname{proposition}{Proposition}{Propositions}

\crefname{equation}{}{}

\newtheorem{theorem}{Theorem}[section]

\newtheorem{proposition}[theorem]{Proposition}
\newtheorem{corollary}[theorem]{Corollary}

\theoremstyle{definition}
\newtheorem{definition}[theorem]{Definition}

\newtheorem{remark}[theorem]{Remark}

\crefname{assumption}{Assumption}{Assumptions}
\crefname{definition}{Definition}{Definitions}
\crefname{corollary}{Corollary}{Corollaries}
\crefname{enumi}{item}{items}

\begin{document}

\title{Spectra and pseudospectra of non-Hermitian Toeplitz operators: Eigenvector decay transitions in banded and dense matrices}
\title{Spectra and pseudospectra of non-Hermitian Toeplitz operators: Localisation transitions in banded and dense matrices}

\author[Y. de Bruijn]{Yannick de Bruijn}
\address{\parbox{\linewidth}{Yannick De Bruijn\\
Department of Mathematics, University of Oslo, Moltke Moes vei 35, 0851 Oslo, Norway.}}
\email{yannicd@math.uio.no}

\author[B. Davies]{Bryn Davies}
\address{\parbox{\linewidth}{Bryn Davies\\
Mathematics Institute, University of Warwick, Coventry CV4 7AL, UK.}}
\email{bryn.davies@warwick.ac.uk}

\author[S. Dupuy]{Sacha Dupuy}
\address{\parbox{\linewidth}{Sacha Dupuy\\
ENSTA, Institut Polytechnique de Paris, 91120 Palaiseau, France}}
\email{sacha.dupuy@ensta.fr}

\author[E. O. Hiltunen]{Erik Orvehed Hiltunen}
\address{\parbox{\linewidth}{Erik Orvehed Hiltunen\\
Department of Mathematics, University of Oslo, Moltke Moes vei 35, 0851 Oslo, Norway.}}
\email{erikhilt@math.uio.no}

\begin{abstract} 
    Using a generalised Floquet-Bloch theory, we present a mathematical method to construct eigenvectors for non-Hermitian Toeplitz operators. We extend the method to both banded Toeplitz operators and those with algebraically decaying, fully dense off-diagonal structure. We present sharp decay estimates for the amplitude of bulk eigenmodes as well as eigenmodes associated with defect eigenfrequencies inside the spectral band gap. The validity of those results is illustrated numerically and we show that banded approximations give poor reconstructions of the dense operators, due to the slow algebraic decay. We apply the insights gained to model the non-Hermitian skin effect in a three-dimensional system of subwavelength resonators, where the corresponding operator exhibits only algebraic decay of off-diagonal entries. We use our approach to demonstrate the fundamental mechanism responsible for the transition between the non-Hermitian skin effect and defect-induced localisation in the bulk.
    \end{abstract}
\maketitle
\vspace{3mm}
\noindent
\textbf{Mathematics Subject Classification (MSC2020): }15A18 
15B05, 
35C20, 
35P25, 
35B34. 

\vspace{3mm}

\vspace{5mm}

\section{Introduction}

Complex band structures and generalisations of Floquet-Bloch theory have proven to be powerful analytical tools for predicting the localisation of eigenvectors in tridiagonal operators \cite{debruijn2024complexbandstructuresubwavelength, debruijn2025complexbandstructurelocalisation, debruijn2025complexbrillouinzonelocalised, yao2018edge, imura2020generalized, song2019non, yang2020non}. In this work we seek to 
extend this framework to arbitrarily $m$-banded and dense Toeplitz operators, where the entries decay algebraically away from the diagonal. We will see that because this algebraic decay is relatively slow, truncating the matrix to a finite number of bands gives a poor approximation of the eigenvectors. In contrast, when the entries decay exponentially away from the diagonal, truncated versions give good approximations of the system, even when they are truncated to the extreme case of being tridiagonal. In such cases, all the off-diagonal elements must be accounted for, which makes the analysis inherently challenging. 
In this work, we develop theory that yields sharp estimates on the off diagonal decay of the inverse of $m$-banded Toeplitz operators, generalising the estimates presented in \cite{DemkoOffDiagonalDecay, JAFFARD1990461}.
Moreover, we show that analysis conducted on a semi-infinite Toeplitz operator may be applied to finite Toeplitz matrices by making a convergent pseudoeigenvector construction. This allows for accurate estimates of the eigenmodes of large but finite dense and $m$-banded Toeplitz matrices.

Dense operators with coefficients that decay only algebraically away from the diagonal are not just a mathematical curiosity, but appear in models for important physical systems. For example, we will apply our analytical results to systems of three-dimensional subwavelength resonators with directional damping terms, which are canonical models for the non-Hermitian skin effect in photonics or phononics. 
The non-Hermitian skin effect is the phenomenon in which all the bulk eigenmodes of a non-Hermitian system are localised at one edge of a finite-sized system. 
We will be particularly interested in resonator arrangements arranged in dimensionally deficient lattices. For example, resonators arranged periodically along one spatial dimension, whereas the physical problem is three-dimensional. Due to this deficiency, resonant modes may propagate in free space in directions orthogonal to the resonator chains and therefore have strong coupling with the far field. It is this long-range coupling that leads to the algebraically slow decay of off-diagonal coefficients. 
This article builds on previous work \cite{ammari2023nonhermitianskineffectthreedimensional}, which considers a three-dimensional system of subwavelength resonators, where the resonators interact with their $k$ nearest neighbours (giving a $k$-banded approximation of the operator). We substantially generalise these results and demonstrate that as the interaction range becomes large, the bulk eigenvectors transition from exponential to algebraic localisation. Moreover, we present a detailed study of defected banded Toeplitz operators. A novel discovery is that eigenmodes are first localised algebraically, locally around the defect site and only transition to exponential localisation further away from the defect site. This novel discovery opens the door to analytically study defect-induced localisation transitions in three-dimensional non-hermitian systems and derive analogue results to \cite{PhysRevB.111.035109} in three dimensions.
We numerically illustrate our findings and compare the localisation of bulk eigenmodes for banded and fully coupled systems.

The paper is organised as follows. In Section \ref{sec: Toeplitz theory}, we present a self-contained analysis of $m$-banded Toeplitz operators with algebraic off-diagonal decay. We recall some results on the spectrum of such operators and characterise the localisation of their eigenvectors. We conclude the section by demonstrating that the eigenvectors of $N$-banded Toeplitz operators provide algebraically good pseudoeigenvectors for finite dense $N\times N$ Toeplitz matrices. We illustrate our results numerically. In Section \ref{sec: Skin effect}, we demonstrate that the complex band structure can be employed to predict the eigenmode localisation in non-Hermitian resonator chains. We illustrate the findings numerically on the non-Hermitian skin effect in three dimensions with long-range couplings.
In Section \ref{Sec: concluding remarks}, we give concluding remarks and discuss future research directions and generalisations of the present work.
The \texttt{Matlab} code that supports the findings of this article is openly available in Section \ref{Sec: Data availability}.

\section{Toeplitz Operators with algebraic off-diagonal decay}\label{sec: Toeplitz theory}
This paper is devoted to the study of the spectra of Toeplitz operators and matrices. A Toeplitz operator is generated by the sequence $\{a_k\}_{k = -\infty}^\infty$ and is given by the semi-infinite matrix,
\begin{equation}\label{eq: def Toeplitz Operator}
    \mathbf{A} = \begin{pmatrix}
        a_0 & a_{-1} & a_{-2} & a_{-3} & \cdots\\
        a_1 & a_0    & a_{-1} & a_{-2} & \cdots\\
        a_2 & a_1 & a_0 & a_{-1} & \cdots \\
        a_3 & a_2 & a_1 &a_0 & \cdots\\
        \vdots & \vdots & \vdots & \vdots & \ddots
    \end{pmatrix}.
\end{equation}
In this paper, we consider only operators and matrices with real-valued coefficients. We say that a Toeplitz operator has \emph{algebraic off-diagonal decay} if
\begin{equation}\label{eq: algebraic off-diagonal decay coefficients}
    \begin{cases}
        a_k = \mathcal{O}(|k|^{-p}), &~k > 0,\\
        a_k = \mathcal{O}(|k|^{-q}), &~k < 0.
    \end{cases}
\end{equation}
We will assume without loss of generality that $p \leq q$, otherwise the adjoint matrix can be considered. 
In \cite{ToeplitzOriginalPaper} Otto Toeplitz proved that a Toeplitz operator \eqref{eq: def Toeplitz Operator} is bounded on $\ell^2$ if and only if there exists a function $f \in L^\infty$ whose sequence of Fourier coefficients is the generating sequence $\{a_k\}_{k = -\infty}^\infty$ for the Toeplitz matrix.
This function is commonly referred to as the \emph{symbol function} of the Toeplitz operator \eqref{eq: def Toeplitz Operator} and is given by the Laurent series
\begin{equation}\label{eq: def symbol function operator}
    f(z) = \sum_{k = -\infty}^\infty a_k z^{-k}.
\end{equation}
We also define an $m$-banded Toeplitz operator $\mathbf{A}_m$ by truncating the sequence:
\begin{equation} \label{eq: def banded Toeplitz Operator}
    (\mathbf{A}_m)_{i,j} = \begin{cases}
        a_{i-j}, &\text{ for } |i-j| < m,\\
        0, &\text{ for } |i-j| \geq m,
    \end{cases}
\end{equation}
in which case the symbol function is given by the finite Laurent polynomial,
\begin{equation}\label{eq: def symbol function banded}
    f_m(z) = \sum_{k = - m}^m a_k z^{-k}.
\end{equation}
In the sequel of this paper, we shall refer to the Toeplitz operator $\vect{A}$, specified in equation \eqref{eq: def Toeplitz Operator}, by the notation $\vect{A} = \vect{T}(f)$ and for the $m$-banded Toeplitz operator \eqref{eq: def banded Toeplitz Operator}, by $\mathbf{A}_m = \mathbf{T}(f_m)$. We note that the symbol function $f_m$ is a special case of $f$, obtained by setting $a_j = 0$ for all $|j| > m$. Consequently, the following results stated for the general symbol function $f$ also apply to $f_m$.
We now define the region of non-trivial winding. We let $\mathbb{T}$ denote the unit circle in $\mathbb{C}$, that is $\mathbb{T} = e^{-\i[0, 2\pi]}$, then the regions of positive and negative winding are defined as 
\begin{align}
    \sigma_{\mathrm{wind}}^+ &:= \bigl\{\lambda \in \C \setminus f(\mathbb{T}) : \operatorname{wind} \bigl(f(\mathbb{T}), \lambda \bigr) >0 \bigr\},\label{def: positive winding}\\
    \sigma_{\mathrm{wind}}^- &:= \bigl\{\lambda \in \C \setminus f(\mathbb{T}) : \operatorname{wind} \bigl(f(\mathbb{T}), \lambda \bigr) <0 \bigr\}\label{def: negative winding},
\end{align}
and the region of non-trivial winding is given by,
\begin{equation}\label{def: winding region}
     \sigma_\mathrm{wind} := \bigl\{\lambda \in \C \setminus f(\mathbb{T}) : \operatorname{wind} \bigl(f(\mathbb{T}), \lambda \bigr) \neq 0 \bigr\} = \sigma_{\mathrm{wind}}^+ \cup \sigma_{\mathrm{wind}}^-.
\end{equation}
We recall the following result of Gohberg \cite[Theorem 1.17]{LargeTruncatedToeplitz} that characterises the spectrum of Toeplitz operators.

\begin{theorem}[Gohberg]\label{thm: Gohberg Spectrum for baned Toeplitz operator}
    The operator $\mathbf{T}(f)$ is Fredholm on $\ell^2$ if  and only if $f(e^{\i\alpha}) \neq 0$ for all $ \alpha \in [0, 2\pi)$, in which case
    \begin{equation}
        \operatorname{Ind}\mathbf{T}(f) = - \operatorname{wind}\bigl(f(\mathbb{T}), 0 \bigr)
    \end{equation}
    and the spectrum is given by
    \begin{equation}
        \sigma\bigl(\mathbf{T}(f)\bigr) =  f(\mathbb{T}) \cup \sigma_\mathrm{wind},
    \end{equation}
    where $\sigma_\mathrm{wind}$ was defined in \eqref{def: winding region}.
\end{theorem}

For later application, we will introduce the limiting spectra of large truncated Toeplitz operators.
To define a finite Toeplitz matrix, we introduce the projection operator,
\begin{align}
    \vect{P}_n: \ell^2(\N, \mathbb{C}) &\to \ell^2(\N, \mathbb{C})\\
    (x_1, x_2, x_3, \dots) &\mapsto (x_1, \dots, x_n, 0, 0, \dots).
\end{align}
A finite Toeplitz matrix is derived from a Toeplitz operator by performing the following truncation,
\begin{equation}\label{eq: truncate operator to matrix}
    \vect{T}_{n}(f) := \vect{P}_{n} \vect{T}(f) \vect{P}_{n}.
\end{equation}

The asymptotic eigenvalue distribution of large, finite truncated Toeplitz matrices is detailed in \cite[Section 5.8]{LargeTruncatedToeplitz}, with the foundational paper being \cite{Schmidt_Spitzer_1960}. The limiting spectrum of a finite but large Toeplitz matrix will be referred to as the \emph{open limit} and is defined as
\begin{equation}\label{def: open spectrum}
    \sigma_{\text{open}}\bigl(\mathbf{T}(f)\bigr) := \lim_{n \to\infty}\sigma\bigl(\mathbf{T}_n(f)\bigr).
\end{equation}
Moreover, by the definition of the open limit, it is not hard to see that $ \sigma_{\text{open}}\bigl(\mathbf{T}(f)\bigr) \subseteq \sigma\bigl(\mathbf{T}(f)\bigr)$.
It is characterised by the following result \cite{Schmidt_Spitzer_1960}.

\begin{theorem}[Schmidt-Spitzer]\label{thm: schmidt spitzer magnitude}
    Let the symbol function $f_m$ be as defined in \eqref{eq: def symbol function banded}. Suppose that the roots $z^p\bigl(f_m(z)-\lambda\bigr)$ are sorted in ascending magnitude, that is, $0< |z_1(\lambda)| \leq |z_2(\lambda)| \leq \dots \leq |z_{2m}(\lambda)|$, then
    \begin{equation}\label{eq: modulus definition}
        \lim_{n \to\infty}\sigma\bigl(\mathbf{T}_n(f_m)\bigr) = \bigl\{ \lambda \in \C ~:~ |z_{m}(\lambda)| = |z_{m+1}(\lambda)| \bigr\}.
    \end{equation}
\end{theorem}
As demonstrated in \cite[Section 3.1]{ammari2024spectra}, tridiagonal Toeplitz operators can always be symmetrized by rescaling the unit circle: if we let $f_r(z) = f(rz)$, there is some $r>0$ such that $\mathbf{T}(f_r)$ is Hermitian and $\sigma(\mathbf{T}(f_r)) = \sigma_\mathrm{open}(\mathbf{T}(f))$.  However, for general $m$-banded Toeplitz operators, we shall see that this condition does not hold: generally, no fixed $r$ exists for which the symbol function $f_r$ generates a Hermitian Toepltiz operator. 

\subsection{Eigenvectors of m-banded Toeplitz operators}\label{sec: Eigenvector construction}Based on the characterisation of the spectrum of Toeplitz operators, we will now present an eigenvector construction for $m$-banded Toeplitz operators. We will use the complex band structure, which generalises the classical Floquet-Bloch band structure by introducing a complex-valued quasiperiodicity. Let $z = e^{\i(\alpha + \i\beta)}$ so that the complex quasiperiodicity $\beta$ captures the magnitude of $z$.

The numerical illustrations in this paper will be carried out for pseudo-Hermitian matrices, i.e. matrices which have real-valued spectra.
This motivates the definition of the complex band structure for real-valued frequencies. 
\begin{definition}
    The \emph{complex band structure} of an $m$-banded Toeplitz operator is parametrised by $\lambda\in\R$ through the $2m$ roots $z_i$ of the polynomial equation
    \begin{equation}\label{eq:poly}
        f_m(z_i) - \lambda = 0,~i \in \{1, 2m\},
    \end{equation}
    such that the different branches are given by $\alpha_i(\lambda)$ and $\beta_i(\lambda)$, where each root satisfies
    \begin{equation}
        z_i = e^{\mathrm{i}(\alpha_i + \mathrm{i}\beta_i)}.
    \end{equation}
\end{definition}

Note that  \eqref{eq:poly} is equivalent to  a polynomial of degree $2m$ and hence has (up to multiplicity) $2m$ roots $z_i$, $i = \{1, \dots, 2m\}$ for any fixed $\lambda\in \mathbb{C}$. The next result shows that any $\lambda \in \mathbb{C}$ is a generalised eigenvalue of $\mathbf{T}(f_m)$ in the sense that we can find $\vect{v}$ (not necessarily in $\ell^2$) satisfying  $\mathbf{T}(f_m)\vect{v} = \lambda \vect{v}$. 

    \begin{theorem}\label{thm: eigvec m banded construction} Let $\lambda \in \C$ such that $f_m(z) = \lambda$, then there exists an eigenvector $\vect{v}$ such that $\mathbf{T}(f_m)\vect{v} = \lambda \vect{v}$, which satisfies    
    \begin{equation}\label{eq: asymptotic eigenvector construction}
        \frac{|\vect{v}^{(i+n)}|}{|\vect{v}^{(i)}|} = \mathcal{O}\bigl(|z_1|^n + \dots + |z_{m+1}|^n \bigr) = \mathcal{O}\bigl( |z_{m+1}|^n \bigr),
    \end{equation}
    where $z_i, ~i\in\{1, 2m\}$ are the $2m$ roots of the equation $f_m(z) = \lambda$ that are sorted in ascending magnitude, i.e. $|z_1| \leq |z_2| \leq \dots\leq|z_{2m}|$.
\end{theorem}

\begin{proof}
The underlying principle hinges on the fact that the symbol function has the same structure as vector matrix multiplication. Let $\vect{v}$ be such that $\mathbf{T}(f)\vect{v} = \lambda \vect{v}$.
Let us make the ansatz $\vect{v} = (1, z^2, z^3, \dots)$ and let $j > m$, then the $j$-th entry of the matrix vector multiplication is of the form
\begin{equation}
    \bigl(\mathbf{T}(f)\vect{v}\bigr)_j = \sum_{i = 1}^\infty \mathbf{T}_{ji}\vect{v}_i
     = \sum_{i = 0}^\infty a_{j-i}z^i 
     = z^j\sum_{i = 0}^\infty a_{j-i}z^{i-j} 
     = z^j \sum_{i = -m}^m a_{-i}z^i = \vect{v}_j f(z).
\end{equation}
This means that, if $z_i$ be a root of $f_m(z_i) = \lambda$, then the vector 
    \begin{equation}
        \vect{v}^i := (\lambda z_i^0, \lambda z_i^1, \lambda z_i^2, \cdots)
    \end{equation}
    satisfies the eigenvalue problem
    \begin{equation}
        T(f)\vect{v}^i = \lambda \vect{v}^i
    \end{equation}
    everywhere but in the first $m$ rows.
    Suppose first that all the roots are distinct, i.e. $z_j \neq z_i, \forall ~ i \neq j$ then a linear combination of $\vect{v}^i$, $i\in \{1, \dots, m+1\}$, yields a vector $\vect{v}$ which is an exact eigenvector for $T(f_m)\vect{v} = \lambda \vect{v}$. This completes the proof under the assumption that the roots are distinct.
    In the case where some of the roots become confluent for some $\lambda$, a confluent Vandermonde matrix approach similar to the proof of \cite[Theorem A.3.]{ammari2023nonhermitianskineffectthreedimensional} may be employed to recover $m+1$ linearly independent eigenvectors $\vect{v}^1, \dots \vect{v}^{m+1}$, which also proves the result in the confluent case.
    From the construction of $\vect{v}$ it is not hard to see that
    \begin{equation}
        \frac{|\vect{v}^{(i+n)}|}{|\vect{v}^{(i)}|} = \mathcal{O}\bigl(|z_1|^n + \dots + |z_{m+1}|^n \bigr) = \mathcal{O}\bigl( |z_{m+1}|^n \bigr),
    \end{equation}
    which completes the proof.
\end{proof}

The estimate for the exponential decay rate of the bulk eigenvectors of non-Hermitian $m$-banded Toeplitz operators \eqref{eq: banded eigenvector decay estimate} is illustrated in Figure \ref{Fig: Complex band structure and decay lengths}, where we plot the complex band structure of a Toeplitz operator and compare with finite truncations.

To estimate eigenvector asymptotics, it is essential to identify the root $z_{m+1}$. For a frequency within the region of positive winding defined in \eqref{def: positive winding}, we observe the following result.

    \begin{theorem}\label{lem: magnitude of roots}
        An eigenvalue $\lambda$ belongs to $\sigma_{\operatorname{wind}}^+$ if and only if there exist $m+1$ roots of the equation $f_m(z) - \lambda = 0$ satisfying $|z_i| < 1$ for $i \in \{0, \dots, m+1\}$.
    \end{theorem}
    \begin{proof}
        The function $f_m(z)-\lambda$ is a meromorphic function in $\C$, which has a pole of order $m$ at $0$. By the argument principle it holds that 
        \begin{equation}\label{eq: winding argument principle}
            \operatorname{wind}\bigl(f_m(\mathbb{T}), \lambda \bigr) = R - P, 
        \end{equation}
        where $R$ is the number of roots of $f_m(z) - \lambda = 0$ with $|z| < 1$ and $P$ is the number of poles of $f_m(z)-\lambda$ for $|z| < 1$, which for the given function is equal to $m$.
        Thus, for each $\lambda \in \sigma_{\operatorname{wind}}^+$, it follows that $R > P = m$.
        On the other hand if $f_m(z) - \lambda = 0$ has $m+1$ roots satisfying $|z_i| < 1$, then by \eqref{eq: winding argument principle} it holds that $\operatorname{wind}\bigl(f_m(\mathbb{T}), \lambda \bigr) \neq 0$, which implies that $\lambda \in \sigma_{\mathrm{wind}}^+$, which completes the proof.
    \end{proof}

As a direct consequence of Theorem \ref{lem: magnitude of roots}, it is not hard to see that for an eigenvector $\vect{v}$ associated to an eigenvalue $\lambda \in \sigma_{\operatorname{wind}}^+$, the vector is exponentially decaying. Theorem \ref{lem: magnitude of roots} also gives a characterisation of the skin effect in terms of the complex band structure by counting the number of positive branches. As we will see in Section \ref{sec: derivation of the discrete Green's function}, for a defect induced frequency outside of the region of positive winding, this may induce a localisation transition from an exponentially skin localised defect mode to an algebraically bulk localised defect mode.

\begin{figure}[htb]
    \centering
    \subfloat[][$m=4$]
  {\includegraphics[width=0.45\linewidth]{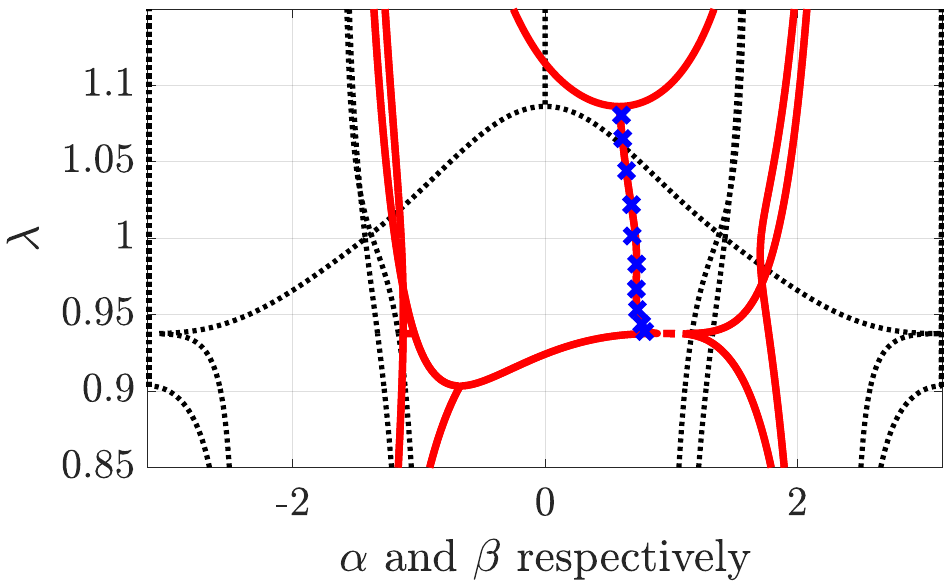}}\quad
    \subfloat[][$m=6$]%
    {\includegraphics[width=0.45\linewidth]{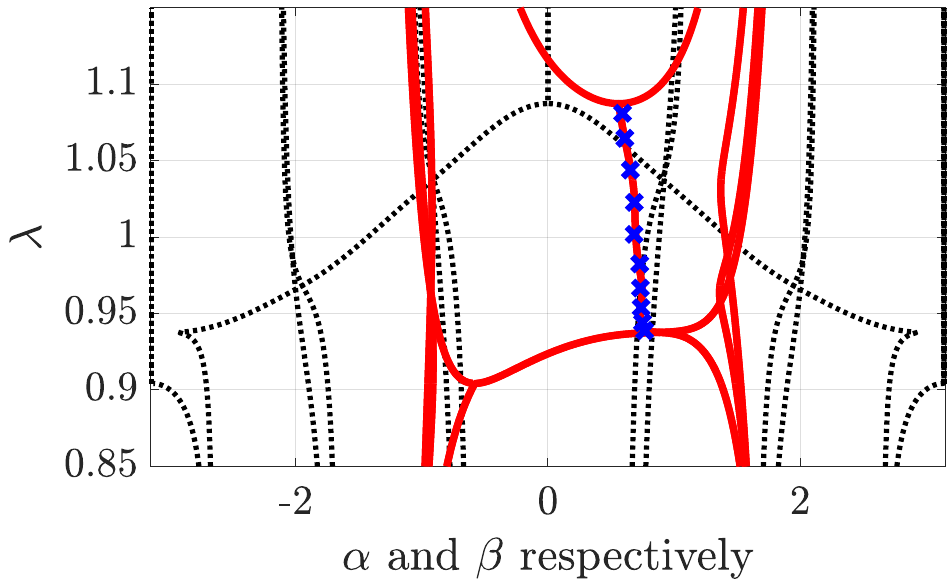}}
    \caption{We plot the complex band structure of two different $m$-banded, pseudo-Hermitian Toeplitz operators. Here, the real part $\alpha$ (black dashed lines) describes the local oscillations, while the imaginary part $\beta$ (red solid lines) describes the exponential decay of the eigenmodes. The decay length of the eigenmodes of a finite truncation of the operator (blue crosses) agree remarkably well with the complex band structure as asserted by Theorem \ref{thm: eigvec m banded construction}. Unlike the tridiagonal case, the decay length of the eigenmodes is frequency dependent.}
    \label{Fig: Complex band structure and decay lengths}
\end{figure}  

\subsection{Spectral behaviour under band truncation}
We will now explore how truncating the symbol function to a finite bandwidth affects the complex band structure, and therefore the localisation properties of the eigenvectors. We consider a full Toeplitz matrix with algebraic decay as defined in \eqref{eq: algebraic off-diagonal decay coefficients}.
An $m$-banded approximation for dense Toeplitz matrices is equivalent to truncating the symbol function, that is, 
\begin{equation}\label{eq: truncated symbol function}
    f(z) = \underbrace{\sum_{j = - m}^m a_j z^{-j} }_{=: f_m(z)}+ \sum_{j = m+1}^\infty (a_jz^{-j} + a_{-j}z^j)
\end{equation}
In a banded approximation, we seek to approximate the roots of full Laurent series $f(z)-\lambda$ by the roots of the finite Laurent polynomial $f_m(z) - \lambda$. Suppose that $z_0$ is a root of the finite Laurent polynomial $f_m(z_0) - \lambda$, then the following bound holds:
\begin{equation}\label{eq: estimate band aprx}
    \lvert f(z_0) - \lambda \rvert \leq \Bigl\lvert f_m(z_0) - \lambda + \sum_{j = m+1}^\infty (a_jz_0^{-j} + a_{-j}z_0^j) \Bigr\rvert 
    = \Bigl\lvert \sum_{j = m+1}^\infty (a_jz_0^{-j} + a_{-j}z_0^j) \Bigr\rvert.
\end{equation}
In the sequel, we will investigate how the complex band structure behaves in the limit as $m \to \infty$.

\begin{corollary}\label{thm: decay rate with number of bands}
    Let $\lambda \in \C$ and let $z_{0,m}$ be any root of $f_m(z)- \lambda$. Then $\lim_{m\to\infty}\lvert z_{0,m} \rvert = 1$.
\end{corollary}
\begin{proof}
    We seek to find $z_{0,m}$ such that $f_m(z_{0,m}) = \lambda < \infty$ and let us define $z_{0,m} := e^{\i(\alpha + \i\beta(m))}$, then,
    \begin{align}
        f(z_0) &= \lim_{m\to\infty}f_m(z_{0,m})\\
        &\geq \lim_{m\to\infty} \sum_{j = 1}^m \Bigl(C_1 \frac{1}{j^p}e^{j \i \alpha}e^{j \beta(m)} + C_2 \frac{1}{j^q}e^{-j \i\alpha}e^{-j\beta(m)} \Bigr)  \\
        &= \lim_{m\to\infty} \Bigl(C_1 \sum_{j = 1}^m  e^{\i j\alpha}e^{-\log(j)p + j \beta(m)} + C_2 \sum_{j=1}^m e^{-\i j\alpha}e^{-\log(j)q - j \beta(m)}\Bigr).
    \end{align}
    By the root test, the above sum is convergent if and only if 
    \begin{equation}\label{eq:betalim}
        \begin{cases}
           \limsup\limits_{m\to\infty} \bigl( -\log(m)p + m\beta(m)\bigr) < 0, \\
            \limsup\limits_{m\to\infty} \bigl(-\log(m)q - m\beta(m) \bigr)< 0.
        \end{cases} 
    \end{equation}
    This means that $\lim\limits_{m\to \infty}\beta(m) = 0$ which proves the result.
\end{proof}

In terms of the complex band structure, as more bands are included, the localisation strength of the eigenvectors decreases. This behaviour is  established by the following theorem.

\begin{theorem}\label{thm: eigenvector banded Toeplitz operator}
    Let $\lambda \in \C$, then there exists an eigenvector $\vect{u}$ of the $m$-banded Toeplitz operator $\mathbf{T}(f_m)$ such that 
    \begin{equation}\label{eq: banded eigenvector decay estimate}
        \frac{|\vect{u}^{(i+j)}|}{|\vect{u}^{(i)}|} = \mathcal{O}(e^{-\beta(m)j}),
    \end{equation}
    where the exponential decay rate is asymptotically given by
    \begin{equation}\label{eq: bound on decay rate}
        \beta(m) = \mathcal{O}\left(\frac{\log(m)p}{m}\right),
    \end{equation}
    where $p$ is the algebraic off-diagonal decay rate of $\mathbf{T}(f_m)$ introduced in \eqref{eq: algebraic off-diagonal decay coefficients}.
\end{theorem}

\begin{proof}
    The eigenvector estimate in \eqref{eq: banded eigenvector decay estimate} is a direct consequence of the eigenvector construction presented in the proof of Theorem \ref{thm: eigvec m banded construction}. By the estimate in \eqref{eq: asymptotic eigenvector construction} it holds that,
     \begin{equation}
        \frac{|\vect{u}^{(i+j)}|}{|\vect{u}^{(i)}|} = \mathcal{O}\bigl(|z_{m+1}|^j \bigr).
    \end{equation}
    By \eqref{eq:betalim} in the proof of Corollary \ref{thm: decay rate with number of bands}, it follows that $\beta(m)$ satisfies the decay rate
    \begin{equation}\label{eq: asymptotic beta}
        \beta(m) = \mathcal{O}\left(\frac{\log(m)p}{m}\right),
    \end{equation}
    which completes the proof.
\end{proof}

Theorem \ref{thm: eigenvector banded Toeplitz operator} asserts that the eigenvectors of $m$-banded Toeplitz operators are always exponentially localised. However, the exponential decay rate decreases towards zero as the bandwidth increases.
The convergence rate of $\beta(m)$ presented in \eqref{eq: bound on decay rate} is numerically verified in Figure \ref{Fig: convergence Floquet algebraic}.

\begin{figure}[h]
    \centering
    \subfloat[][Algebraic off-diagonal decay rate $p = 1.1$.]%
    {\includegraphics[width=0.48\linewidth]{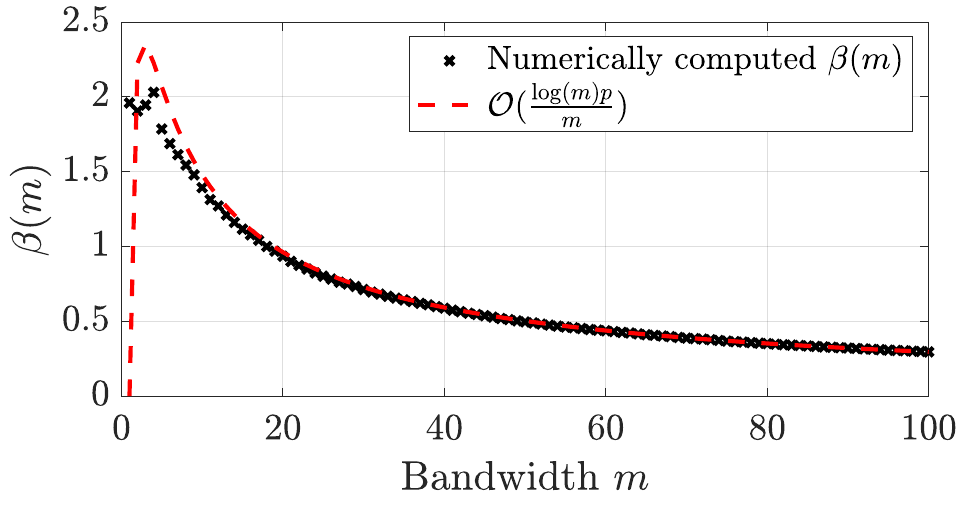}}\quad
    \subfloat[][Algebraic off-diagonal decay rate $p = 1.6$.]%
    {\includegraphics[width=0.48\linewidth]{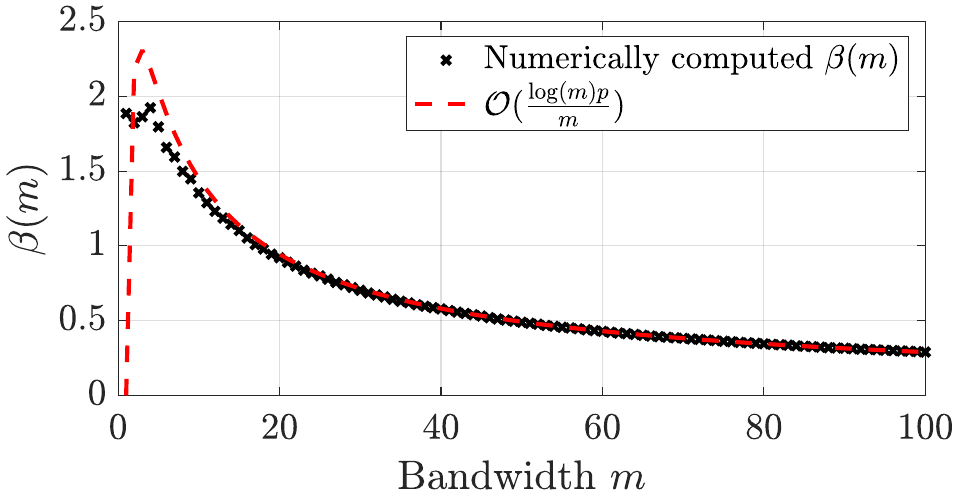}}
    \caption{The asymptotics of $\beta(m)$ are correctly predicted by \eqref{eq: asymptotic beta} and  $\beta(m) \to 0$ as $m\to \infty$.}
    \label{Fig: convergence Floquet algebraic}
\end{figure}

\begin{remark}\label{rem: exponential off diagonal decay}
    We would like to mention that the case where the entries of the Toeplitz matrix decay exponentially away from the diagonal, the truncation error introduced in  \eqref{eq: truncated symbol function} will be exponentially small. Therefore a banded approximation of the full Toeplitz operator yields a very accurate prediction of the spectral properties of the full operator.
\end{remark}

\subsection{Connection between the complex band structure and the inverse of a Toeplitz operator.}\label{sec: derivation of the discrete Green's function}
In order to excite a resonant frequency outside of the spectrum, i.e. $\lambda \not\in \sigma\bigl(\mathbf{T}(f)\bigr)$, the underlying system needs to experience an outside forcing, which is given by, possibly a linear combination of, $\delta_i$. In a physical setting this external forcing is often given my the form of a defect. The defected resonance problem becomes of the form,
\begin{equation}
    \mathbf{T}(f)\vect{u}-\lambda \vect{u} = \delta_i,
\end{equation}
which yields that the defect eigenmode is given by (potentially through a linear combination of),
\begin{equation}\label{eq: defect eigenmode resolvent}
    \vect{u} = \bigl (\mathbf{T}(f) - \lambda \Id \bigr)^{-1}\delta_i.
\end{equation}
Since the complex band structure yields a decay estimate on the entries of $\vect{u}$, we may achieve a bound on any column of the operator $\bigl(\vect{T}(f) - \lambda \Id\bigr)^{-1}$. For this reason, we may describe defect modes either by achieving an off-diagonal decay estimate of the operator $\bigl(\vect{T}(f) - \lambda \Id\bigr)^{-1}$ or by employing the complex band structure directly to the defect eigenmode $\vect{u}$. In the first case $\vect{u}$ is also commonly referred to as the discrete Green's function.

At this point, it is worth mentioning that the off-diagonal entries of the inverse of a banded matrix are expected to decay exponentially. For general banded matrices, not necessarily Toeplitz, the following bound may be achieved \cite[Proposition 2.2]{DemkoOffDiagonalDecay}.

\begin{theorem}[Demko et al.]\label{thm: demko off diagonal}
    Let $\mathbf{A}$ be $m$-banded, bounded, and invertible on $\ell^2$. Let $\kappa$ be the condition number of $\mathbf{A}$ and let $q = (\sqrt{\kappa}-1) / (\sqrt{\kappa} + 1)$, then
    \begin{equation}\label{eq: demko estimate}
        \lvert \mathbf{A}^{-1}(i,j) \rvert \leq C e^{\frac{2}{m}\log(q)|i-j|}.
    \end{equation}
\end{theorem}

Although the bound established in Theorem \ref{thm: demko off diagonal} holds universally for any banded matrix, it is generally not sharp. The matrices which we will subsequently treat have the distinction that apart from being banded, they are Toeplitz and the off diagonals entries decay algebraically. We will therefore present a similar but sharper bound using the complex band structure. 

A notable attribute of Toeplitz operators exhibiting algebraic off-diagonal decay is that the decay rate is preserved under upon matrix inversion. This result was initially established in \cite[Proposition 3]{JAFFARD1990461}

\begin{theorem}[Jaffard]\label{thm: jaffard off diagonal}
    If $\mathbf{A}$ is an invertible operator with bounded inverse and suppose $\alpha > 1$ such that
    \begin{equation}
        \lvert \mathbf{A}(i,j)\rvert \leq C \bigl\lvert 1 + |i-j| \bigr\rvert^{-\alpha},
    \end{equation}
    then
    \begin{equation}
        \lvert \mathbf{A}^{-1}(i,j)\rvert \leq C \bigl\lvert 1 + |i-j| \bigr\rvert^{-\alpha}.
    \end{equation}
\end{theorem}

From standard spectral theory, the resolvent $(\mathbf{T}(f)-\lambda\Id)^{-1}$ is well-defined if and only if $\lambda \not\in\sigma\bigl(\mathbf{T}(f)\bigr)$. By the characterisation of the spectrum in Theorem \ref{thm: Gohberg Spectrum for baned Toeplitz operator}, Jaffard's estimate for the resolvent always holds for defect eigenfrequencies outside of the closure of $\sigma_{\mathrm{wind}}$. This is the reason why Theorem \ref{thm: jaffard off diagonal} will be of particular interest for studying the localisation properties of defect eigenmodes.

For $m$-banded Toeplitz operators, it is clear that asymptotically the exponential bound in \eqref{eq: demko estimate} always wins, but there is the possibility that for entries close to the diagonal, the off-diagonal decay is algebraic and only later transitions to exponential. This transition is numerically illustrated in Figure \ref{Fig: Jaffard Demko CBS Estimates}.

\begin{figure}[h!]
    \centering
    \subfloat[][Non-Hermitian $8$-banded Toeplitz matrix with algebraic off-diagonal decay.]
    {\includegraphics[width=0.45\linewidth]{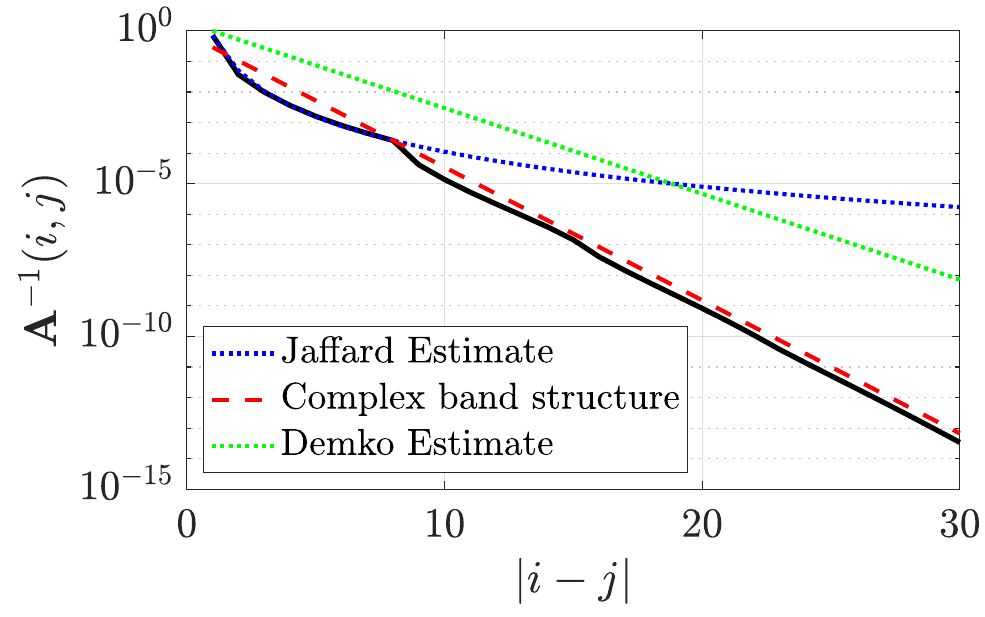}}\quad
    \subfloat[][Non-Hermitian $20$-banded Toeplitz matrix with algebraic off-diagonal decay.]%
    {\includegraphics[width=0.45\linewidth]{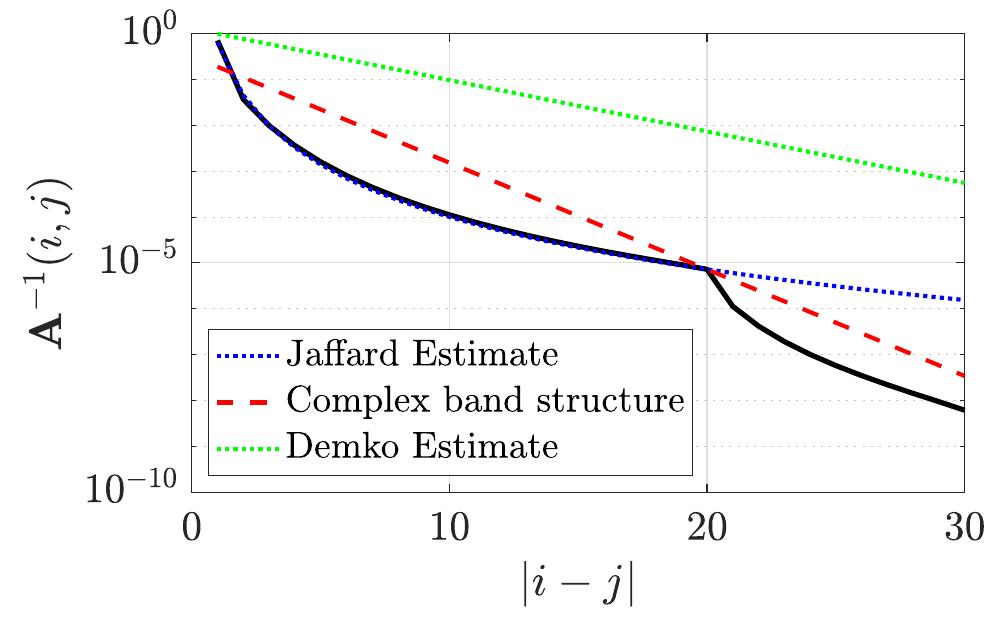}}
    \caption{The algebraic bound established by Jaffard correctly bounds the first entries of the eigenvector. As discussed previously, both Demko's exponential bound Theorem \ref{thm: demko off diagonal} and  the bound by the complex band structure are valid. However, the exponential bound achieved by the complex band structure is much tighter.}
    \label{Fig: Jaffard Demko CBS Estimates}
\end{figure}

Naturally, we achieve two bounds on the eigenmodes, which are competing. When one now considers the $m$-banded limit $m\to\infty$, by Theorem \ref{thm: eigenvector banded Toeplitz operator}, the asymptotic exponential decay rate converges to $0$, and therefore loses to the Jaffard decay estimate. Consequently, in the dense Toeplitz case, the eigenmodes are algebraically localised, which agrees with previous results \cite[Proposition 3]{JAFFARD1990461}.
In Figure \ref{fig:SpectralDecomposition} we recapitulate the main findings and illustrate the defect induced decay transition at different defect eigenfrequencies.

\begin{figure}[h]
    \centering
    \includegraphics[width=0.9\linewidth]{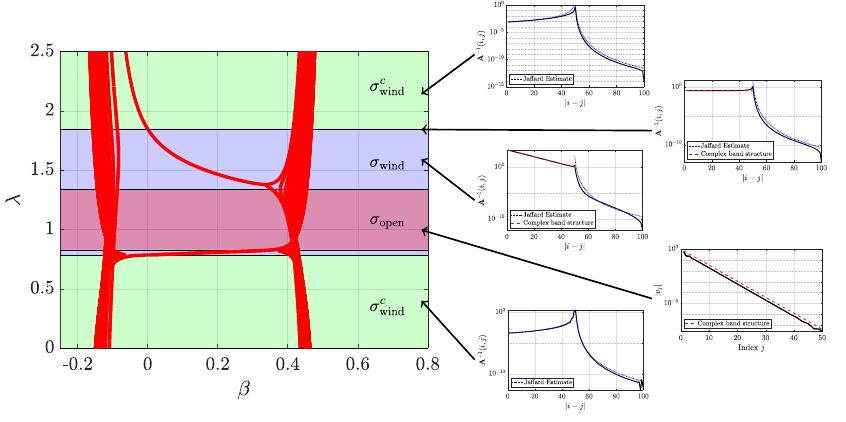}
    \caption{The red lines denote the complex band structure for a $50$-banded non-Hermitian Toeplitz matrix with algebraic off-diagonal decay. The red region denotes the open spectrum $\sigma_{\text{open}}$ defined in \eqref{def: open spectrum}. Eigenvectors for eigenvalues in the open spectrum are exponentially skin localised, and the exponential decay rate is predicted by the complex band structure. The green shaded region, $\sigma_{\mathrm{wind}}^\mathsf{c}$ denotes the eigenvalues for which Jaffard's theorem is applicable. The blue region, which fully contains the spectrum, denotes the winding region where Jaffard's estimate is not applicable. At the transition between $\sigma_{\mathrm{wind}}$ and $\sigma_{\mathrm{wind}}^\mathsf{c}$, the eigenmodes are constant leading up to the defect. Consequently the qualitative decay behaviour is accurately captured by the regions $\sigma_{\mathrm{wind}}$, $\sigma_{\mathrm{wind}}^\mathsf{c}$ and $\sigma_{\text{open}}$}
    \label{fig:SpectralDecomposition}
\end{figure}

\subsection{Pseudospectra for banded Toeplitz matrices.}\label{sec: pseudospectra}
This section is devoted to the study of finite banded and dense Toeplitz matrices.
As pointed out in \cite{REICHEL1992153}, the spectra of non-Hermitian systems are highly sensitive to small perturbations. The solution is offered by forming $\varepsilon$-pseudoeigenvectors through the truncation of localised Bloch modes. This leads us to introduce pseudospectra. For a detailed discussion, see \cite[Section 1-2]{trefethen.embree2005Spectra}.

\begin{definition}\label{def: pseudospectrum}
    Let \(\varepsilon > 0\). Then, \(\lambda \in \mathbb{C}\) is an \(\varepsilon\)-pseudoeigenvalue of \(\mathbf{A}\in \mathbb{C}^{N\times N}\) if one of the following conditions is satisfied:
    \begin{enumerate}[(i)]
        \item \(\lambda\) is a proper eigenvalue of \(\mathbf{A} + \mathbf{E}\) for some \(\mathbf{E} \in \mathbb{C}^{N\times N}\) such that \(\lVert \mathbf{E} \rVert \leq \varepsilon\);
        \item $\lVert (\mathbf{A}-\lambda \Id) \mathbf{u} \rVert < \varepsilon$ for some vector $u$ with $\lVert \vect u \rVert = 1$;
        \item $\lVert (\mathbf{A}- \lambda \Id)^{-1} \rVert^{-1} \leq \epsilon$.
    \end{enumerate}
    The set of all $\varepsilon$-pseudoeigenvalues of $\vect A$, the $\varepsilon$-pseudospectrum, is denoted by $\sigma_{\epsilon}(\vect A)$. If some non-zero $\vect u$ satisfies $\lVert (\lambda\Id - \mathbf{A}) \mathbf{u} \rVert < \varepsilon$, then we say that $\vect u$ is an $\varepsilon$-pseudoeigenvector of $\mathbf{A}$. By the norm equivalence in finite dimensions, any norm may be used.
\end{definition}

Let us denote by $\vect{T}_N(f_m) \in \R^{N \times N}$ the $m$-banded Toeplitz matrix which results from truncating the operator $\vect{T}(f_m)$ to the first $N$ entries as in \eqref{eq: truncate operator to matrix} and let $\lambda_N$ be the associated eigenvalue. In this manner, note that $\mathbf{T}_N(f_N)$ is a finite and dense $N\times N$ Toeplitz matrix.

\begin{proposition}\label{prop: Pseudoeigenvector construction}
    Let $\vect{T}_N(f) \in \R^{N \times N}$ be a dense non-Hermitian Toeplitz matrix with algebraic off-diagonal decay as in \eqref{eq: algebraic off-diagonal decay coefficients} and let $\vect{v}_N$ be a truncated eigenvector of the $N$-banded Toeplitz operator. Then 
    $\vect{v}_N$ is an $\varepsilon_N$-pseudoeigenvector, with
    \begin{equation}
        \varepsilon_N = \mathcal{O}\left(\frac{N^{1-2p}(N^p-1)}{N^{p/N}-1}\right).
    \end{equation}

\end{proposition}

\begin{proof}
    The truncated eigenvector $\vect{v}_N$ satisfies the eigenvalue problem $\vect{T}(f) \vect{v}_N = \lambda_N\vect{v}_N$ only in the first row. The error in the other rows is bounded as follows,
    \begin{equation}\label{eq: dense Toeplitz pseudospectrum}
        \left \lVert \left( \vect{T}_N(f) - \lambda_N\Id_N\right) \vect{v}_N \right\rVert_1 = \left \lVert \begin{pmatrix}
        0 \\
        -a_{-(N+1)}\vect{v}_N^{N+1}\\
        \vdots \\
        -\sum_{j = 1}^N a_j\vect{v}_N^{N+i}
        \end{pmatrix} \right\rVert_1 = \left\lVert \Bigl(
            \sum_{i=1}^j a_{-(N-j+i)}\vect{v}^{N+i}
        \Bigr)^\top_{j = \{1,\dots, N\}} \right\rVert_1 =\varepsilon_N
    \end{equation}
    A bound on $\varepsilon_N$ is achieved by the fact that the coefficients $a_i$ are bounded and by Theorem \ref{thm: eigenvector banded Toeplitz operator}, we know that $\vect{v}_N^i = \mathcal{O}(e^{-\beta(N) i})$. By Corollary \ref{thm: decay rate with number of bands} the decay rate depends on the number bands and is bounded by $\beta(N) \leq \frac{\log(N)p}{N}$. The residual norm in \eqref{eq: dense Toeplitz pseudospectrum} can therefore be estimated as follows:
    \begin{align}
    \varepsilon_N  &\leq \sum_{j = 1}^N \sum_{i = 1}^j e^{-\ln(N-j+i)p}e^{-\beta(N)(N+i)} \\
    & = e^{-\beta(N)N} \sum_{j=1}^N \sum_{i=1}^j e^{-\ln(N+i-j)p-\beta(N)i} \\
    &\leq N^{-p} \sum_{j=1}^N (N+1-j)^{-p} \sum_{i=1}^j e^{-\beta(N)i} \label{eq: crudest bound} \\
    & \leq N^{-p}  \sum_{j = 1}^N  (N+1-j)^{-p} N\frac{N^{-p}(N^p-1)}{N^{p/N}-1} \\
    &= N^{-p} N\frac{N^{-p}(N^p-1)}{N^{p/N}-1} (N+1) \sum_{i = 1}^N j^{-p} \\
    &\leq \zeta(p) \frac{N^{1-2p}(N^p-1)}{N^{p/N}-1},\label{eq: final bound}
\end{align}
where $\zeta(p)$ is the Riemann Zeta function and the bound on $\beta(N)$ was established in \eqref{eq: asymptotic beta}. 
\end{proof}
Note that the bound in \eqref{eq: crudest bound} is crude but holds for general $p>1$. Numerically, for $p \geq 5$ the bound becomes quite sharp.
For asymptotocially large $N$, the bound in \eqref{eq: final bound} decays algebraically.  As a consequence, by truncating the eigenvectors of a banded Toeplitz operator, one can construct algebraically good $\varepsilon_N$-pseudoeigenvectors for a Toeplitz matrix.

In \Cref{Fig: Pseudospectrum convergence}, we numerically verify the convergence of the $\varepsilon_N$-pseudospectrum established in Proposition \ref{prop: Pseudoeigenvector construction}. This result has important implications regarding spectral properties of large, finite and dense Toeplitz matrices. In particular, the eigenvectors of such matrices are qualitatively very different to the eigenvectors of corresponding infinite dense Toeplitz operators. Instead, a finite matrix of size $N$ should be compared to a \emph{banded} infinite Toeplitz matrix with $N$ bands.

\begin{figure}[ht]
    \centering
    \subfloat[][Computation performed for non-Hermitian Toeplitz matrix with off-diagonal decay rates $p = 2$ and $ q = 3.5$.]
    {\includegraphics[width=0.45\linewidth]{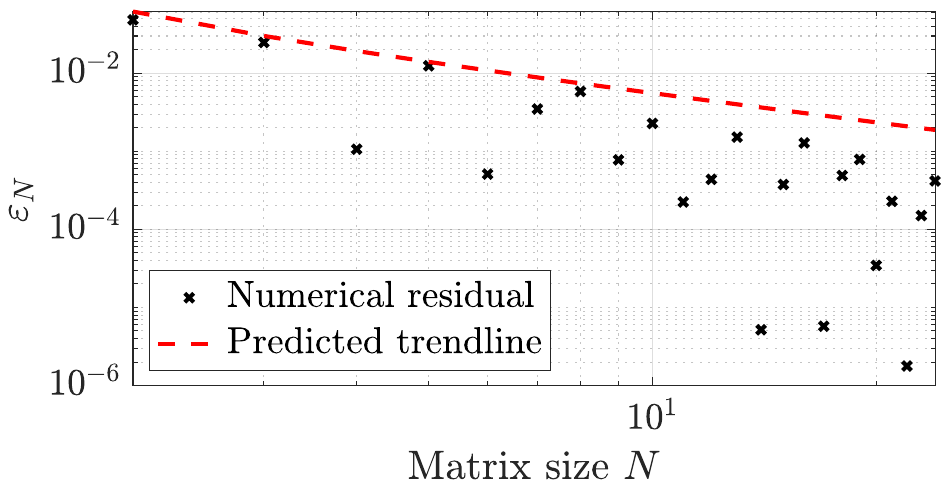}}\quad
    \subfloat[][Computation performed for non-Hermitian Toeplitz matrix with off-diagonal decay rate $p = q = 1.4$.]%
    {\includegraphics[width=0.45\linewidth]{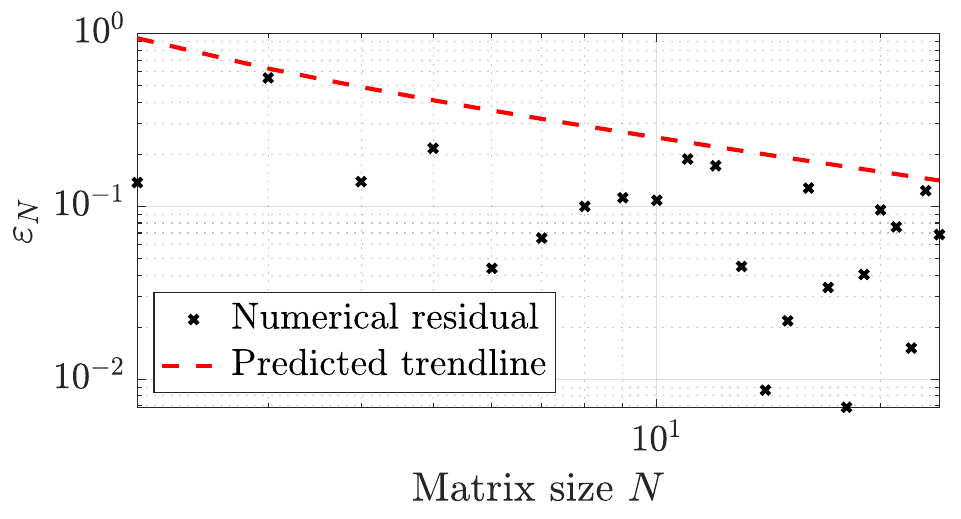}}
    \caption{ We numerically verify the bound established in Proposition \ref{prop: Pseudoeigenvector construction} for the convergence of the $\varepsilon_N$-pseudospectrum for dense non-Hermitian Toeplitz matrices. The pseudospectrum is algebraically convergent, which is to be expected as the decay rate of the eigenvalues decreases as $\mathcal{O}\bigl(\log(N)/N\bigr)$.}
    \label{Fig: Pseudospectrum convergence}
\end{figure}  

\subsection{Proper non-Hermitian Operators}\label{sec: non-Hermitian operator}
In this section, we illustrate that a variation of  the complex band structure may also be defined for $m$-banded Toeplitz matrices whose spectrum is no longer purely real. Crucially, the estimate from Theorem \ref{thm: eigenvector banded Toeplitz operator} still applies. We consider the roots of $f_m(z)-\lambda = 0$; we sort the roots in ascending order $|z_1| \leq \dots |z_m| \leq |z_{m+1}| \leq \dots \leq |z_{2m}|$ and plot $\beta$ where  $e^{-\beta} = |z_{m+1}|$. Clearly, for a frequency $\lambda \in \sigma_{\mathrm{wind}}^+$ it holds that $|z_{m+1}| < 1$, hence $\beta > 0$.
In Figure \ref{Fig: Complex band structure with complex frequency}, we numerically illustrate the complex band structure for complex-valued frequencies.

\begin{figure}[ht]
    \centering
    \subfloat[][$4$-banded Toeplitz matrix with algebraic off-diagonal decay.]
    {\includegraphics[width=0.45\linewidth]{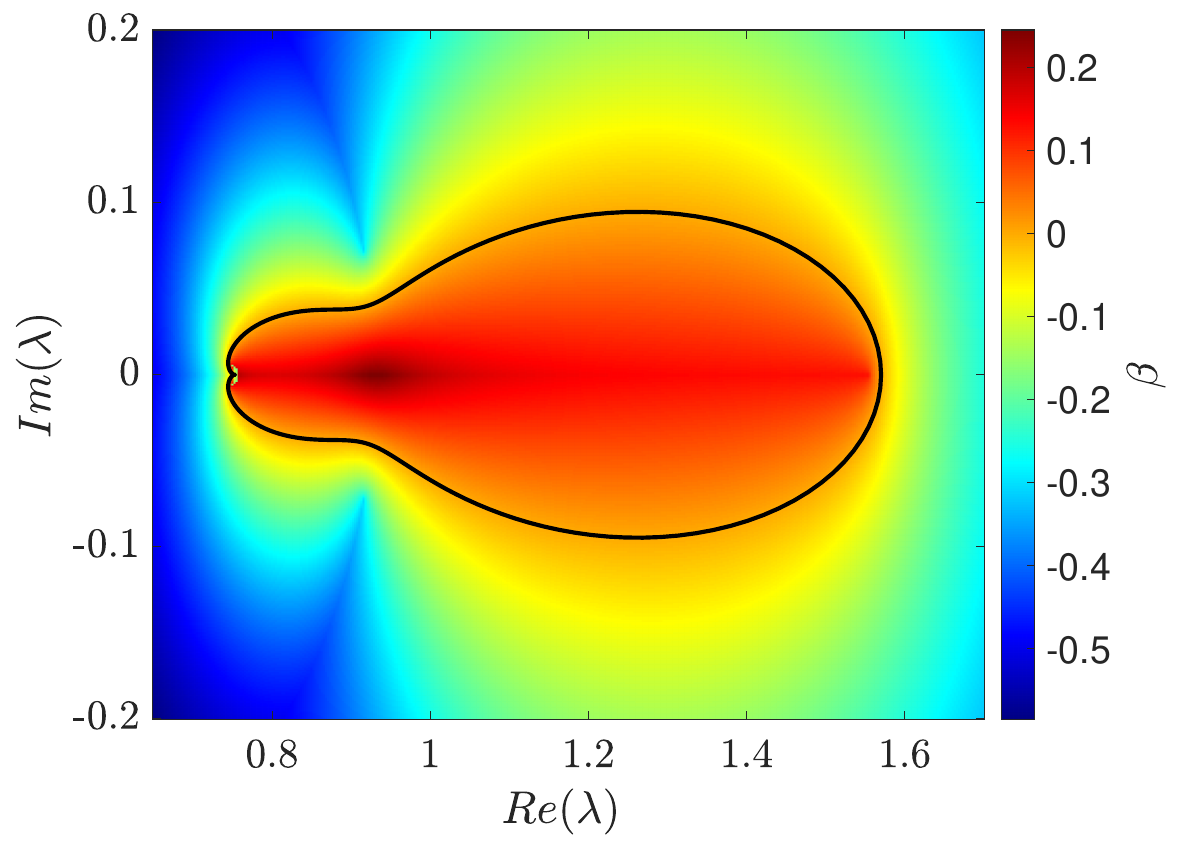}}\quad
    \subfloat[][$4$-banded Toeplitz matrix with exponential off-diagonal decay.]%
    {\includegraphics[width=0.45\linewidth]{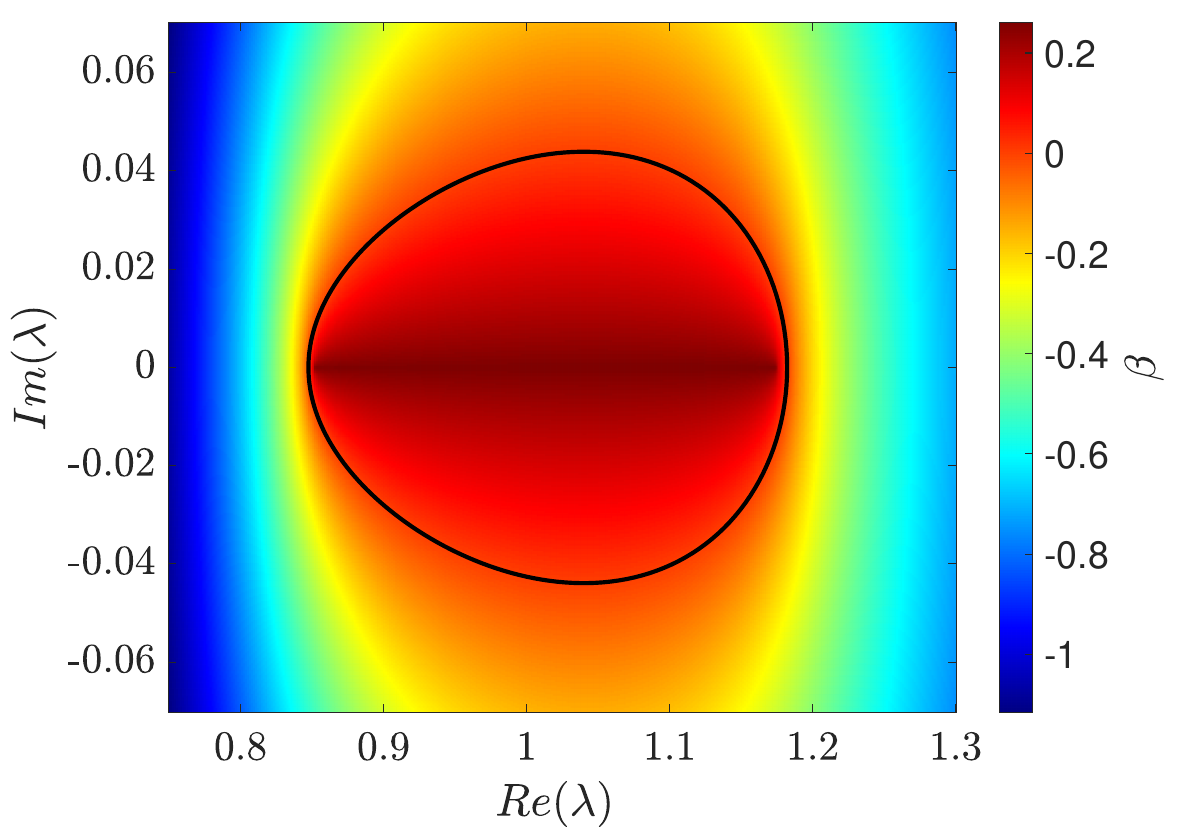}}
    \caption{The complex band structure for $\lambda \in \C$. The solid black line is the curve traced out by the symbol function $f(\mathbb{T})$. For eigenvalues contained within this curve, the complex Floquet parameter $\beta$ is positive, leading to exponentially localised eigenmodes. The decay strength is the strongest for $\lambda\in \R$ within the winding region.}
    \label{Fig: Complex band structure with complex frequency}
\end{figure}

\subsection{Complex band structure for Hermitian Toeplitz matrices}
For a Hermitian matrix, the sequence that generates the Toeplitz matrix is symmetric, that is, $a_i = a_{-i}$, $\forall i \in \N$. The symbol function \eqref{eq: def Toeplitz Operator} reduces to
\begin{equation}
    f_m(z) = a_0 + \sum_{j = 1}^m a_j(z^{-j} + z^j),
\end{equation}
which for $z = e^{\i(\alpha + \i\beta)}$ is the trigonometric function,
\begin{equation}
    f_m(e^{\i(\alpha + \i \beta)}) = a_0 + 2 \sum_{j = 1}^m a_j\bigl(\cos(\alpha k)\cosh(\beta k) - 2\i\sin(\alpha k)\sinh(\beta k) \bigr).
\end{equation}
Firstly as Hermitian matrices have real spectrum, it is advisable to restrict the symbol function to have a real-valued image, that is $\Im\bigl(f_m(e^{\i(\alpha + \i\beta)})\bigr) = 0$. An example of admissible arguments that generate a real values symbol function is illustrated in Figure \ref{Fig: Admissible Quasiperiodicities}. 

In the case of Hermitian Toeplitz operators, the spectral statement from Theorem \ref{thm: Gohberg Spectrum for baned Toeplitz operator} my be improved by the following result \cite{SpectraHermitianToeplitz}.
\begin{theorem}\label{thm: spectrum Hermitian Toeplitz}
    Let $a_n \in \ell^2$, $f$ be defined as in \eqref{eq: def symbol function operator}, $a_i = a_{-i}$ and let $r \leq f_m(\mathbb{T}) \leq R$, where $\mathbb{T}$ denotes the one-dimensional torus, then
    \begin{enumerate}
        \item The spectrum $\sigma\bigl(\mathbf{T}(f_m)\bigr)$ is the closed interval $[r, R] \subseteq \R$.
        \item If $f_m(\mathbb{T})$ is not constant, that is, if $r < R$, then the point spectrum of $\mathbf{T}(f_m)$ is empty.
    \end{enumerate}
\end{theorem}
In particular, under the assumptions of Theorem \ref{thm: spectrum Hermitian Toeplitz} and $f(\mathbb{T}) \neq \text{const.}$, the operator $\mathbf{T}(f)$ has only essential spectrum.
For Hermitian banded Toeplitz matrices it is therefore reasonable to define the band gap as the complement of the spectrum $\R \setminus \sigma\bigl(\mathbf{T}(f)\bigr)$. 

\begin{figure}[htb]
    \centering
    \subfloat[][$1$-banded (tridiagonal) Hermitian Toeplitz operator.]%
    {\includegraphics[width=0.45\linewidth]{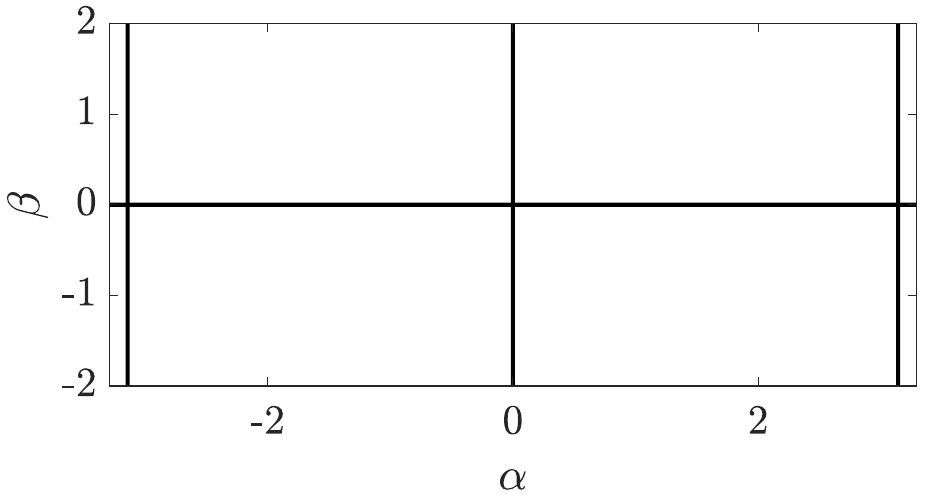}}\quad
    \subfloat[][$8$-banded Hermitian Toeplitz operator.]%
    {\includegraphics[width=0.45\linewidth]{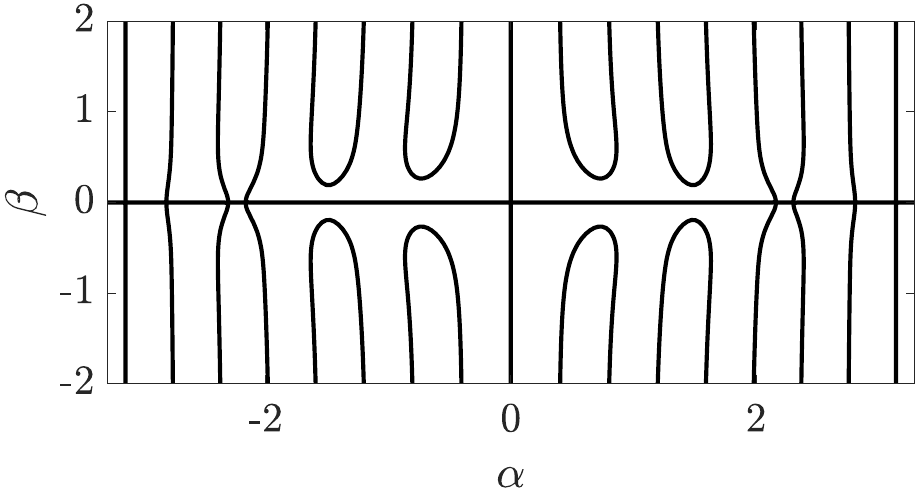}}
    \caption{Contour plot in $\alpha$ and $\beta$ for which $\Im\bigl(f_m(e^{\i(\alpha + \i\beta)})\bigr) = 0$. The figure was generated for banded Hermitian Toeplitz Matrices with an algebraic off-diagonal decay rate of $p = 1.8$.}
    \label{Fig: Admissible Quasiperiodicities}
\end{figure}
For tridiagonal Toeplitz matrices, that is $1$-banded matrices, it has been shown that the complex Floquet parameter $\beta$ is constant \cite[Theorem 2.3.]{debruijn2025complexbandstructurelocalisation}. 
The fact that there is no analogue of \cite[Theorem 2.3.]{debruijn2025complexbandstructurelocalisation} in the banded case shows that it is not possible to find closed formulas for the gap functions. For the remainder, we therefore have to revert to numerical calculations as illustrated in Figure \ref{Fig: Admissible Quasiperiodicities}.

\section{Non-Hermitian Skin effect with long range coupling}\label{sec: Skin effect}

In the present section, we study a three-dimensional system of
finitely many subwavelength resonators with a non-Hermitian coupling. We will be particularly interested in the localisation properties of the eigenmodes in the presence of defects in the structure.
We recall the results from \cite{ammari2023nonhermitianskineffectthreedimensional} which introduced a discrete approximation of the eigenmodes and eigenfrequencies of the system in terms of the eigenvectors and eigenvalues of the so-called \emph{gauge capacitance matrix} $\mathcal{C}^\gamma_N$, which is a dense non-Hermitian Toeplitz matrix with algebraic off-diagonal decay.
We will employ the insights gained on non-Hermitian Toeplitz operators from Sections \ref{sec: derivation of the discrete Green's function} and \ref{sec: pseudospectra} to construct algebraically good pseudoeigenvectors for the finite capacitance matrix and to explain the defect induced localisation transition.

\subsection{Finite resonator chain}

We recall the derivation of the gauge capacitance matrix associated with a three dimensional finite system of spherical resonators, developed in \cite{ammari2023nonhermitianskineffectthreedimensional}. We let $D\subset \R^3$ be the ball of radius $R\in (0,1/2)$. Subsequently, we examine a one-dimensional sequence described by $\mathcal{D} := \bigcup_{i=1}^N D_i$, where each $D_i$ is defined as the translated domain $D_i := D + (i,0,0)^\top$, with the superscript $\top$ signifying the transpose operation. A resonant frequency $\omega\in\C$ is characterised by $\Re(\omega)>0$, and the associated non-trivial eigenmode $u$ that satisfies the equation. 
\begin{equation}
    \label{helmholtz}
    \begin{cases}
    \ds \Delta u + k^2u = 0  & \text{in } \R^3 \setminus \overline{\mathcal{D}}, \\
    \ds \Delta u + k_i^2 u +\gamma \partial_1 u = 0  & \text{in } D_i, \quad i=1,\dots,N, \\
    \ds  u|_{+} -u|_{-}  = 0  & \text{on } \p \mathcal{D}, \\
    \ds \left.\delta \frac{\p u}{\p \nu} \right|_{+} - \left.\frac{\p u}{\p \nu} \right|_{-} = 0 & \text{on } \p \mathcal{D}, \\
    \ds u \text{ satisfies an outgoing radiation condition}.    
    \end{cases}
\end{equation}
where $k_i = \omega/v_i$ is known to be the wave number and $v_i$ the speed within the inclusions. 
In this context, $\nu$ represents the outward normal vector to the boundary $\partial D$. The notations $|_{+}$ and $|_{-}$ indicate the limit values approached from the exterior and interior of the domain $\mathcal{D}$, respectively. The parameter $\delta >0$, representing a non-dimensional contrast in materials, is considered to be small, ensuring the system operates within a high-contrast regime. 
It is important to note that parameters $k$ and $\omega$ will be treated interchangeably due to their asymptotic equivalence. Furthermore, the first-order directional derivative, characterised by a non-zero coefficient $\gamma$, signifies an imaginary gauge potential present within the inclusions. This term disrupts the time-reversal symmetry of the system, serving as the essential mechanism behind the condensation phenomena anticipated in this analysis.

We focus on the subwavelength regime, in which we seek the resonant frequencies and corresponding eigenmodes of the resonator system $\mathcal{D}$. This regime is characterised by the condition that $\omega \to 0$ as $\delta \to 0$.

\subsection{Asymptotic resonance expansion}

In a high contrast low frequency regime it was established in \cite{ammari2023nonhermitianskineffectthreedimensional}, that the capacitance matrix accurately captures the leading order asymptotics of the resonance problem.
We will briefly recall the main results and refer the reader to \cite{ammari2023nonhermitianskineffectthreedimensional} for a detailed analysis.
A fundamental solution to the operator with an imaginary gauge potential $\Delta+k^2 +\gamma \partial_{x_1}$ is given by 
\begin{equation}\label{green_skin}
    G_\gamma^k(x) = -\frac{ \mathrm{exp}({-{\gamma}x_1/{2}+\mathrm{i}\sqrt{k^2-\gamma^2/4}|x|})}{4\pi |x|}.      
\end{equation}
We define the single-layer potential associated to the fundamental solution $G_{\gamma}^k$ by
\begin{equation}\label{eq:S}
    {\tilde{\mathcal{S}}}_{D,\gamma}^k:\phi \in L^2(\partial D)\mapsto \int_{\partial D}G_\gamma^k(x-y)\phi(y)\mathrm{d}\sigma(y).
\end{equation}
The operator  ${\tilde{\mathcal{S}}}_{D,\gamma}^0 : L^2(\partial D) \rightarrow L^2(\partial D)$ is injective and hence has a left inverse. A solution to the scattering problem \eqref{helmholtz} may be represented as
\begin{equation}
\label{eq:ux_layerpotential}
    u(x)=
    \begin{cases}
        \tilde{\mathcal{S}}_{\mathcal{D},\gamma}^{k_i}[\psi], \ \ \ &x\in {D_i},\\
        \mathcal{S}_\mathcal{D}^k[\phi], \ \ \ & x\in \mathbb{R}^3\backslash \mathcal{D},
    \end{cases}
\end{equation}
for some unknown densities, $\psi, \phi \in L^2(\partial D)$. Here, $\mathcal{S}_\mathcal{D}^k$ denotes the Helmholtz single-layer potential, defined analogously as in \eqref{eq:S} but using the Helmholtz Green's function $G_0^k$.
We now introduce the gauge capacitance matrix,  which allows us to reduce the problem of finding the  subwavelength eigenfrequencies and eigenmodes to a finite-dimensional eigenvalue problem. 
\begin{definition}[Gauge capacitance matrix] 
The gauge capacitance matrix is defined entry-wise for $i,j \in \{1, \dots, N\}$ as
\begin{equation} \label{capacitancedef}
    \left(\mathcal{C}_N^\gamma\right)_{i,j} = -\frac{\delta v_i^2}{\int_{D_i}e^{\gamma x_1}\ \mathrm{d}x} \int_{\partial D_i} e^{\gamma x_1}(\mathcal{S}^0_{\mathcal{D}})^{-1}[\chi_{\partial D_j}](x) \, \mathrm{d}x \in \C^{N \times N},
\end{equation}
 where $\chi_{\partial D_i}$ denotes the indicator function of $\partial D_i$.
\end{definition}
The subwavelength resonances may be recovered using the Fundamental Theorem of Capacitance for a three dimensional non-Hermitian resonator chain \cite[Theorem 4.4]{ammari2023nonhermitianskineffectthreedimensional}. 
\begin{theorem} \label{thm:approx}
The $N$ subwavelength eigenfrequencies $\omega_i$ of (\ref{helmholtz}) satisfy, as $\delta \rightarrow 0$,
\begin{equation} \label{approxlambda}
\omega_n= \sqrt{\lambda_n} + O(\delta), 
\end{equation}
where $(\lambda_n)_{1\leq n\leq N}$ are the eigenvalues of $\mathcal{C}_N^\gamma$.
Let $v_n$ be the eigenvector of $\mathcal{C}_N^\gamma$ associated to $\lambda_n$. Then the normalised resonant mode $u_n$ associated to the resonant frequency $\omega_n$ is given by
\begin{equation} \label{eigenmodeapprox}
    u_n(x) = 
    \begin{cases}
        v_n \cdot \mathbf{S}^{k_n} (x) + O(\delta) \ \ &\text{in}\ \mathbb{R}^3\setminus \bar{\mathcal{D}}, \\
         v_n \cdot \tilde{\mathbf{S}}^{k_n}(x) + O(\delta) \ \ & \text{in}\ D_i,
    \end{cases}
\end{equation}
where 
\begin{equation}
\label{eigenvector:skin}
    \tilde{\mathbf{S}}^{k_n}(x) = \begin{pmatrix}
        \tilde{\mathcal{S}}_{\mathcal{D},\gamma}^{k_n} [\psi_1](x)\\ \ldots \\  \tilde{\mathcal{S}}_{\mathcal{D},\gamma}^{k_n} [\psi_N](x)
    \end{pmatrix},
\end{equation}
with $\psi_i = (\tilde{\mathcal{S}}_{\mathcal{D},\gamma}^0)^{-1}[\chi_{\partial D_i}]$ for $i=1,\ldots,N$.
\label{thmgauge}
\end{theorem}

We now consider an infinitely periodic system with one resonator in a unit cell spaced along the lattice $\Lambda$, which corresponds to am infinite sequence of uniformly spaced resonators.
The quasiperiodic Green's function outside the resonators is defined as
\begin{equation}\label{eq:Galph}
    G^{\alpha, k}(x, y):=-\sum_{m \in \Lambda} \frac{e^{\mathrm{i} k|x-y-m|}}{4 \pi|x-y-m|} e^{\mathrm{i} \alpha \cdot m},
\end{equation}
together with the quasiperiodic single layer potential
\begin{equation}
    \mathcal{S}_D^{\alpha, k }[\phi](x):=\int_{\partial D} G^{\alpha, k}(x, y) \phi(y) \mathrm{d} \sigma(y), \quad x \in \mathbb{R}^3,
\end{equation}
for $\phi \in L^2(\partial D)$. 
The lattice sum in \eqref{eq:Galph} converges uniformly for $x$ and $y$ in compact sets of $\R^3$, $x\neq y$,  and $k \neq |\alpha + q|$ for all $q\in \Lambda^*$.
The ``real-space'' capacitance representation is defined through the inverse Floquet–Bloch transform and can be written as
\begin{equation}\label{eq:real_space_infinite_capacitance}
    \mathcal{C}^\gamma_{i,j}=\frac{1}{|Y^*|}\int_{Y^*}\hat{C}^{\alpha,\gamma}e^{-\i\alpha(i-j)}\d\alpha.
\end{equation}
Here, $Y^*:= \R /\Lambda^*$ denotes the Brillouin zone associated to the lattice $\Lambda$. Finally, $\alpha$ is the quasiperiodic Floquet parameter and $\hat{C}^{\alpha,\gamma}$ is the quasiperiodic capacitance representation defined as
\begin{equation}
    \hat{C}^{\alpha,\gamma} = -\frac{\delta v^2}{\int_D e^{\gamma x_1}\d x}\int_{\partial D} e^{\gamma x_1}(\mathcal{S}^{\alpha,0}_D)^{-1}[\chi_{\partial D}](x)\d \sigma(x).
\end{equation}
It is clear from the definition \eqref{eq:real_space_infinite_capacitance} that $\mathcal{C}^\gamma$ is a Toeplitz matrix with symbol function $\hat{C}^{-\alpha, \gamma}$.

\subsection{Toeplitz structure and non-Hermiticity of the gauge capacitance matrix}
Here, we explain how we relate the non-Hermiticity and the asymptotic properties of $\mathcal{C}^\gamma_N$ to the general and abstract framework of dense or banded Toeplitz operators. Moreover, we illustrate how spectral analysis of Toeplitz operators developed in Section \ref{sec: Toeplitz theory} can effectively predict the localisation properties of eigenmodes in a three dimensional resonator chain.

It is known since the works of Maxwell that the coefficient $\left(\mathcal{C}_N^\gamma\right)_{i,j}$ can be interpreted as the coupling between resonators $i$ and $j$. In the present setting of a $3D$ resonator chain, the coupling between the resonators is relativity strong, only decaying algebraically in space.  
Following \cite[Proposition 5.9.]{ammari2023nonhermitianskineffectthreedimensional} the entries of the capacitance matrix are known to decay at least algebraically, such that
\begin{equation}
    \bigl\lvert (\mathcal{C}_N^\gamma)_{i,j} \bigr\rvert \leq \frac{\delta K}{|i-j|},
\end{equation}
for some constant $K$ which is independent of $N$.
The first order derivative term in \eqref{helmholtz} introduces an asymmetry in the coupling between the resonators along the axis of the lattice. It is responsible for the non-Hermiticity of $\mathcal{C}^\gamma_N$. We still emphasise that this asymmetry does not lead to different decay rates of the matrix entries away from the main diagonal. Numerically, we find that the off-diagonal decay rates introduced in \eqref{eq: algebraic off-diagonal decay coefficients} satisfy $p \approx q \approx 1.4$.

We now need to explain in which sense the convergence of the finite system to the infinite occurs. Let $\mathcal{D}_N$ denotes a finite chain of $N$ equidistant resonators and let $\mathcal{C}^\gamma_N$ be the corresponding gauge capacitance matrix defined in \eqref{capacitancedef}. Now consider $\mathcal{D}_M$ a larger chain than $\mathcal{D}_N$ and $\mathcal{C}^\gamma_M$ the associated capacitance matrix. We extract the embedded $N\times N$ central block of $\mathcal{C}^\gamma_M$ and denote it $\tilde{\mathcal{C}}^\gamma_N$. Then this truncated capacitance satisfies the following convergence result \cite[Theorem 5.2.]{ammari2023nonhermitianskineffectthreedimensional}, for all $i,j=1,\cdots, N$, 
$$\lim_{M\to \infty}(\tilde{\mathcal{C}}^\gamma_N)_{i,j}=(\mathcal{C}^\gamma)_{i,j},$$
where $\mathcal{C}^\gamma$ is the Toeplitz operator defined in \eqref{eq:real_space_infinite_capacitance} corresponding to a proper infinite resonator chain.
Effectively, as the matrix size increases, possible edge effects due to the open boundary become negligible. In other words, large capacitance matrices may be accurately approximated by proper Toeplitz matrices.
We subsequently approximate the capacitance matrix by a Toeplitz matrix in the following way.
We first compute a sufficiently large gauge capacitance matrix of size $M$, for example $M=100$. Then, we extract the central block of size $N\times N$ and compute a regularisation to get an exact Toeplitz matrix, which means that we replace each entry with the average of the associated diagonal.

\begin{figure}[h]
    \centering
    \subfloat[][Convergence of our Toeplitz construction from gauge capacitance matrices in relative Frobenius norm.]%
    {\includegraphics[width=0.485\linewidth]{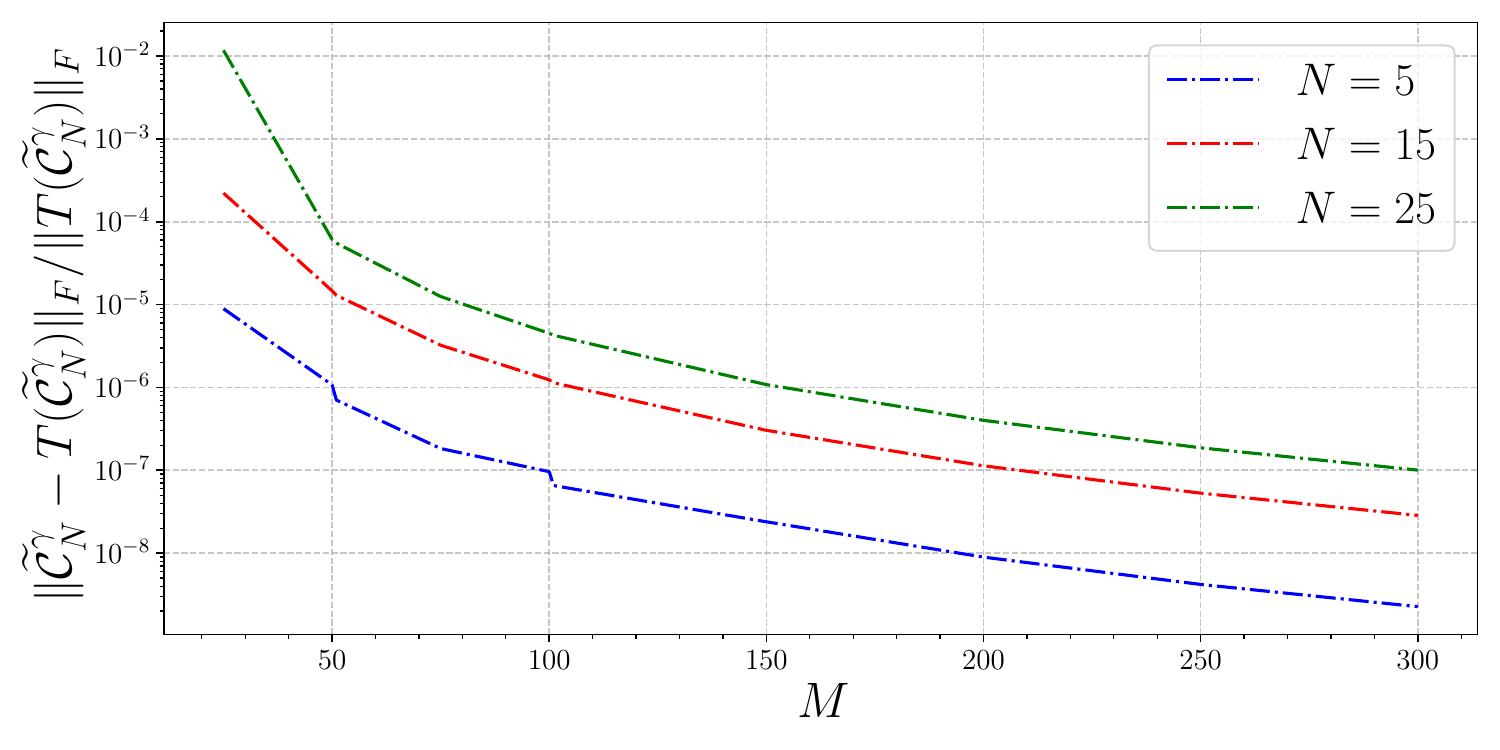}}\quad
    \subfloat[][Scattering and algebraic fit of the upper and lower entries of the capacitance matrix Toeplitz approximation.]%
    {\includegraphics[width=0.485\linewidth]{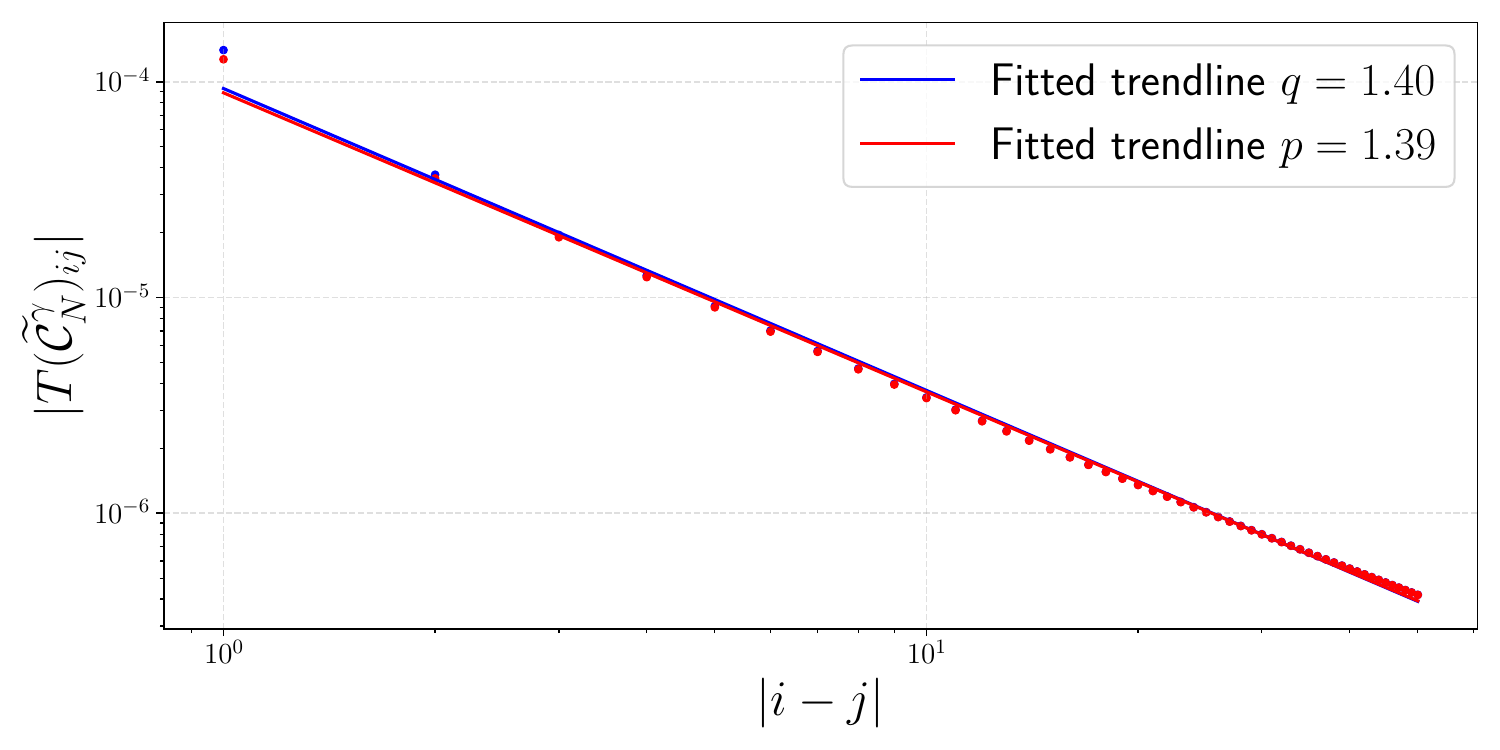}}
    \caption{Our construction method is consistent with the abstract framework developed above. For sufficiently large $M$, the extracted block $\widetilde{\mathcal{C}}^\gamma_N$ is close to its best Toeplitz approximation. It allows to get rid of the edge effects.}
    \label{Fig:Toeplitz construction of capacitance matrices}
\end{figure}

\subsection{Defect induced localisation in finite resonator chain}\label{Sec: defect modes}
In the present section, we will investigate how the localisation properties change in the presence of defects. The resonator chain comprises a defect, which is under the form of a changed wave speed. Notably, the position of the resonators within the chain, as well as their geometry is unaffected by the defect. Such a chain is depicted in Figure \ref{fig:3D-resonators-1D-chain}.

\begin{figure}[H]
    \centering
    \includegraphics[width=0.60\linewidth]{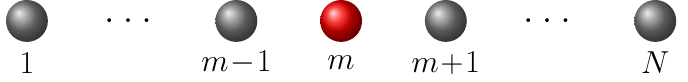}
    \caption{Finite $1D$ chain of $3D$ spherical resonators with a defect at the index $m$. The defect comprises a changed wave speed and leaves the geometry of the chain unaffected.}
    \label{fig:3D-resonators-1D-chain}
\end{figure}

The changed wave speed on the $m$-th resonator is captured by the diagonal matrix $\Bf \in \R^{N\times N}$, given by
\begin{equation}\label{eq: defect matrix}
    \Bf := \begin{cases}
        1, & i \neq m, \\
        1 + \eta, & i = m,
    \end{cases}
\end{equation}
for some constant $\eta \in (-1, \infty)$. Localised modes are associated to eigenfrequencies $\omega^2 \in \sigma(\Bf\mathcal{C}_N^\gamma)$ such that $\omega^2 \not\in \sigma(\mathcal{C}_N^\gamma)$. Defective materials are characterised by the following eigenvalue problem,

\begin{equation}\label{eq: defect resonance problem}
    \Bf \mathcal{C}_N^\gamma \vect{u} - \lambda \vect{u} = 0
\end{equation}

\begin{theorem}\label{thm: discrete Greens function defect mode}
    Consider a defected structure with a defect in the $m$-th resonator. Let $\omega^2 \in \sigma(\Bf \mathcal{C}_N^\gamma)$ but $\omega^2 \not\in\sigma(\mathcal{C}_N^\gamma)$, then up to scaling the discrete Green's function
    \begin{equation}
        \vect{u} = (\mathcal{C}_N^\gamma - \omega^2\Id)^{-1}\delta_m
    \end{equation}
    is an eigenmode of \eqref{eq: defect resonance problem}.
\end{theorem}

\begin{proof}
    Note that \eqref{eq: defect matrix} may alternatively be written as
    \begin{equation}
        \Bf = \Id + \eta e_me_m^\top
    \end{equation}
    A direct computation yields,
    \begin{equation}
        \Bf \mathcal{C}_N^\gamma \vect{u} - \omega^2\vect{u} = 0 \Leftrightarrow \mathcal{C}_N^\gamma \vect{u} + \eta \delta_m (\mathcal{C}_N^\gamma \delta_m)\vect{u} - \omega^2 \vect{u} = 0\Leftrightarrow \vect{u} = -(\mathcal{C}_N^\gamma - \omega^2\Id)^{-1} \delta_m(\mathcal{C}_N^\gamma \delta_m)\vect{u}.
    \end{equation}
    Since $(\mathcal{C}_N^\gamma \delta_m)\cdot\vect{u}\in \R$ is a scalar, the defect eigenmode is given by
    \begin{equation}
        \vect{u} = - K(\mathcal{C}_N^\gamma - \omega^2\Id)^{-1}\delta_m
    \end{equation}
    for some constant $K \in \R$.
\end{proof}

We will now proceed to illustrate the decay transition numerically within a resonator chain, similar to the simulations carried out in Figure \ref{fig:SpectralDecomposition}.

\begin{figure}[htb]
    \centering
    \subfloat[][$8$-banded approximation of the capacitance matrix.]%
    {\includegraphics[width=0.45\linewidth]{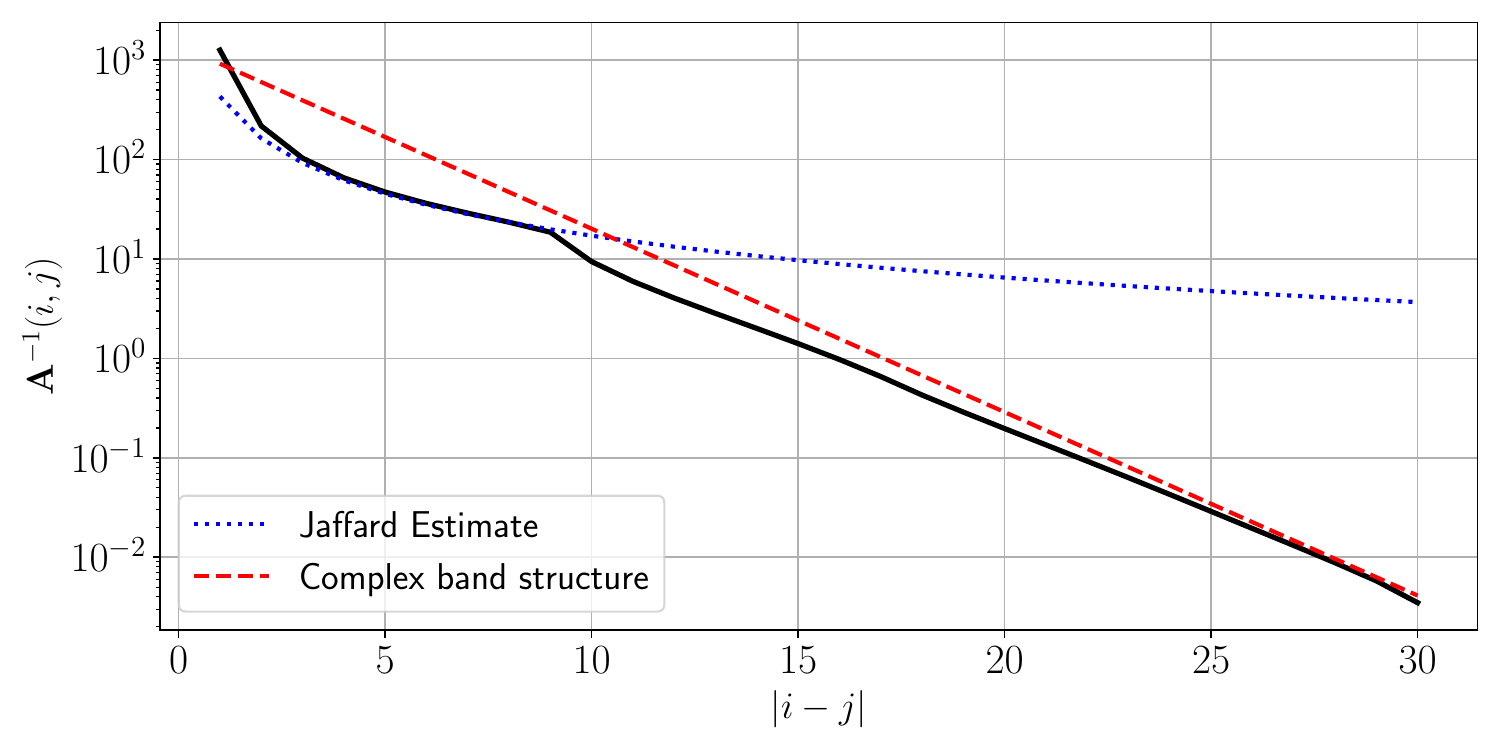}}\quad
    \subfloat[][$20$-banded approximation of the capacitance matrix.]%
    {\includegraphics[width=0.45\linewidth]{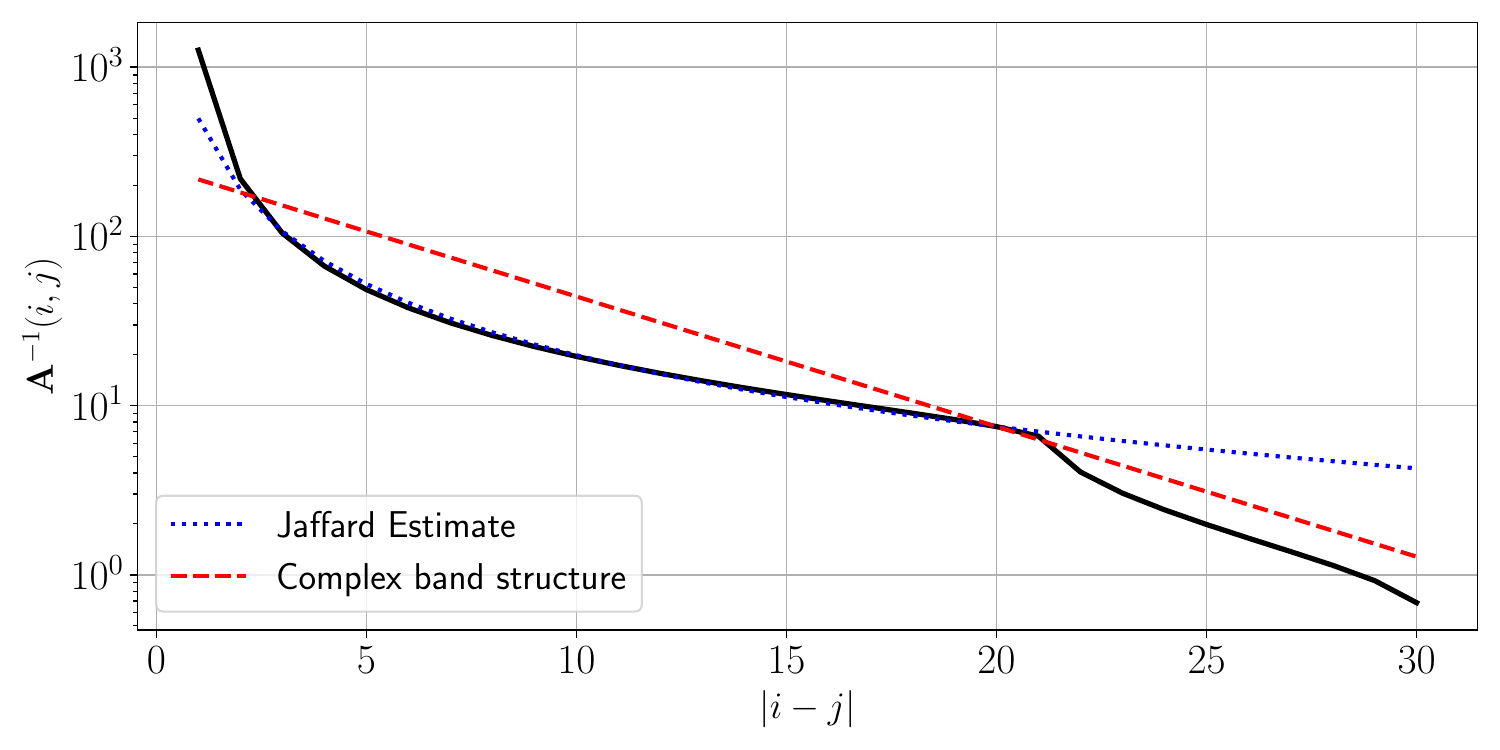}}
    \caption{Defect eigenmode corresponding to \eqref{eq: defect resonance problem}. As predicted by Theorem \ref{thm: eigenvector banded Toeplitz operator} the exponential decay rate predicted by the complex band structure decreases, the more bands are considered. Consequently, Jaffard's estimate from Theorem \ref{thm: jaffard off diagonal} stays valid for longer, illustrating that the localisation observed strongly depends on the number bands considered.}
    \label{Fig: Bandwidth difference}
\end{figure}

We will now proceed to illustrate the defect induced decay transition in three dimensional resonator chains and numerically illustrate the defect eigenmode behaviour of \eqref{eq: defect resonance problem}.

\begin{figure}[h]
    \centering
    \includegraphics[width=0.95\linewidth]{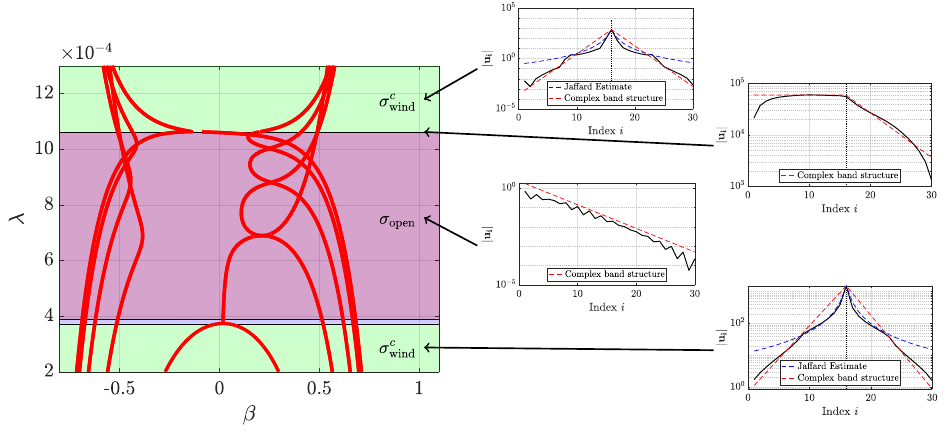}
    \caption{Similar illustration as in Figure \ref{fig:SpectralDecomposition}. The red lines denote the complex band structure for a $8$-banded Toeplitz approximated capacitance matrix. The red region denotes the open spectrum $\sigma_{\text{open}}$ defined in \eqref{def: open spectrum}. Eigenvectors for eigenvalues in the spectrum are exponentially skin localised, and the exponential decay rate is predicted by the complex band structure. The green shaded region $\sigma_{\mathrm{wind}}^\mathsf{c}$ denotes the eigenvalues for which Jaffard's theorem is applicable. 
    The winding region $\sigma_{\mathrm{wind}}$ which fully contains $\sigma_{\text{open}}$, almost fully agrees with $\sigma_{\text{open}}$ and is therefore almost not visible in the Figure.
    At the transition between $\sigma_{\text{open}}$ and $\sigma_{\text{open}}^\mathsf{c}$, the eigenmodes are constant leading up to the defect.}
    \label{fig:SpectralDecompositionCapacitance}
\end{figure}

Note that the computations in Figure \ref{fig:SpectralDecompositionCapacitance} were carried out for an $8$-banded capacitance matrix approximation. For a defect in the green region, the eigenmodes become exponentially localised beyond eight entries from the defect site. Consequently, in a fully coupled system, the defect mode would exhibit only algebraic localisation around the defect site, as in Figure \ref{fig:SpectralDecomposition}, highlighting that the coupling length drastically affects the quantitative defect eigenmode behaviour.

\section{Concluding Remarks}\label{Sec: concluding remarks}
Our article illustrates the localisation properties observed in non-Hermitian systems with long-range coupling length. We show that the complex band structure naturally characterises  localisation transitions of such systems, where defect eigenmodes transition from skin localised (governed by the complex band structure) to bulk localised (governed by  Jaffard type estimates). We present a convergent method to construct algebraically good pseudoeigenvectors for $m$-banded and dense Toeplitz matrices with algebraic off-diagonal decay. This highlights the nature of the eigenmodes: a finite matrix of size $m\times m$ exhibits exponentially decaying eigenmodes akin to the $m$-banded (infinite) Toeplitz operator, although the decay rate vanishes as $m$ grows.

For this paper, many of the numerical illustrations assumed a real-valued limiting spectra $\sigma_{\text{open}}$, though the analytical result hold for any complex-valued eigenvalues. In an upcoming paper, we will give a more concise characterisation along with a criterion for when the limiting spectra of banded Toeplitz matrices are real, which is of particular relevance since many Toeplitz operators arising from physical models possess real limiting spectra.

\section{Data availability} \label{Sec: Data availability}
The \texttt{Matlab} code for the numerical experiments developed in this work is openly available in the following repository:
\url{https://github.com/yannick2305/BandedToeplitz}.

\printbibliography

\end{document}